\documentclass[reqno,centertags,11pt]{amsart}
\usepackage{amsthm}
\usepackage{amsmath,amsfonts}
\usepackage{geometry}
\usepackage{amssymb}
\usepackage{theoremref}
\usepackage{mathrsfs}
\usepackage{enumerate}
\usepackage{hyperref}
\usepackage{xcolor}
\usepackage{kotex}

\newtheorem{thm}{Theorem}
\newtheorem{lem}[thm]{Lemma}

\theoremstyle{definition}
\newtheorem{dfn}[thm]{Definition}
\newtheorem{rmk}[thm]{Remark}

\numberwithin{thm}{section}
\numberwithin{equation}{section}

\makeatletter
\newcommand*{\rom}[1]{\expandafter\@slowromancap\romannumeral #1@}
\makeatother

\newcommand{\N}{{\mathbb{N}}}
\newcommand{\R}{{\mathbb{R}}}

\title[Optimal regularity results in Sobolev-Lorentz spaces]{Optimal regularity results in Sobolev-Lorentz spaces   for linear elliptic   equations with $L^1$- or measure data}
\date{\today}

\author[H. Kim]{Hyunseok Kim}
\author[Y.-R. Lee]{Young-Ran Lee}
\author[J. Ok]{Jihoon Ok}

\address{Department of Mathematics, Institute for Mathematical and Data Sciences, Sogang University, 35 Baekbeom-ro, Mapo-gu, Seoul 04107, Republic of Korea.}
\email{kimh@sogang.ac.kr, younglee@sogang.ac.kr, jihoonok@sogang.ac.kr}

\begin{document}

\begin{abstract}
It has been  well known that if  $\Omega$ is a bounded $C^1$-domain in $\R^n,\ n \ge 2$, then   for every Radon measure $f$ on $\Omega$ with finite total variation,  there exists a unique weak solution   $u\in W_0^{1,1}(\Omega )$ of the Poisson equation   $-\Delta u=f$ in $\Omega$ satisfying  $\nabla u \in L^{n/(n-1),\infty}(\Omega;\R^n )$. In this paper, optimal regularity properties of the solution $u$ are established   in Sobolev-Lorentz spaces $L_{\alpha}^{p,q}(\Omega )$   of order $ \alpha$ less than but arbitrarily close to $2$. More precisely, for any $0 \le \alpha<1$, we show that $u\in L_{\alpha+1}^{p(\alpha),\infty}(\Omega )$, where  $p(\alpha )= n/(n-1+\alpha )$. Moreover, using  an embedding result for Sobolev-Lorentz spaces $L_{\alpha}^{p,q}(\Omega )$ into classical   Besov spaces $B_\alpha^{p,q}(\Omega )$, we deduce  that  $u\in B_{\alpha+1}^{p(\alpha),\infty}(\Omega )$.
Indeed, these regularity results are proved for solutions of the Dirichlet problems for more general linear elliptic equations with    nonhomogeneous boundary data.

On the other hand, it is    known that if $\Omega$ is of class $C^{1,1}$,  then for each $G\in L^1 (\Omega ;\R^n )$ there exists a unique very weak solution   $v\in L^{n/(n-1),\infty}  (\Omega )$ of   $-\Delta v= {\rm div}\, G$ in $\Omega$ satisfying the   boundary condition $v=0$ in some sense. We prove that   $v$ has the optimal regularity property,  that is,   $v\in L_{\alpha}^{p(\alpha),\infty}(\Omega )\cap B_{\alpha}^{p(\alpha),\infty}(\Omega )$ for every $0 \le \alpha < 1$.
This regularity result is  also  proved for more general equations  with nonhomogeneous boundary data.
\end{abstract}


\maketitle

\section{Introduction}

In this paper, we study   the  following Dirichlet problems  for   linear elliptic equations on a bounded    domain $\Omega$ in   $\R^n ,\ n \ge 2  $:
\begin{equation}\label{eq:D-1-introd}
\left\{\begin{alignedat}{2}
- \Delta u + b \cdot \nabla u +cu  & =f & \quad & \text{in }\Omega,\\
  u & ={u_D} & \quad & \text{on }\partial\Omega
\end{alignedat}\right.
\end{equation}
and
\begin{equation}\label{eq:D-2-introd}
\left\{\begin{alignedat}{2}
- \Delta v -{\rm div}\, (v  b)+ cv  & ={\rm div}\, G & \quad & \text{in }\Omega,\\
  v & =v_D & \quad & \text{on }\partial\Omega ,
\end{alignedat}\right.
\end{equation}
where   $f$ is an integrable function or more generally a signed Radon measure on $\Omega$ with finite total variation, and $G = (G_1 , \ldots , G_n )$ is an integrable vector field on $\Omega$.

Elliptic equations with $L^1$- or measure data have been  studied extensively by many authors  \cite{BBGGPV,BBC,BO,BG,BG2,BS,DMOP,KO,Min07} for linear and nonlinear   problems. A quite well-known result
 (see, e.g., \cite{BBGGPV})
is that  if  the coefficients $b$ and $ c$ are bounded and the domain $\Omega$ is of class $C^{1}$, then  for   every Radon measure $f$ on $\Omega$ with finite total varation, there exists a unique weak solution $u \in W_0^{1,1}(\Omega )$ of (\ref{eq:D-1-introd}) with $u_D=0$ satisfying $\nabla u \in L^{n/(n-1), \infty }(\Omega; \R^n )$.
In addition, if $\Omega$ is of class $C^{1,1}$, then it can be shown by a duality argument (see, e.g.,  \cite{KO}) that
for every $G\in L^1 (\Omega ; \R^n )$, there exists a unique very weak solution $v \in L^{n/(n-1), \infty }(\Omega )$ of (\ref{eq:D-2-introd}) with  $v_D=0$ satisfying the boundary condition $v=0$ in some very weak sense.
Here  $L^{p,q}(\Omega )$ denotes  the  Lorentz space on the domain $\Omega$, with $L^{p,\infty}(\Omega )$ often referred to as the weak $L^p$-space or the Marcinkiewicz space on $\Omega$.

The purpose of  this paper is to establish optimal     regularity results for the weak solution $u$ and the very weak solution $v$ of the problems  (\ref{eq:D-1-introd}) and (\ref{eq:D-2-introd}), respectively,  with $L^1$- or measure data, in the framework of Sobolev-Lorentz spaces as well as classical  Besov spaces.

\subsection{Optimal regularity for the whole space case}\label{sec1.1}

To explore optimal  regularity of solutions of (\ref{eq:D-1-introd}) and (\ref{eq:D-2-introd}) with $L^1$- or measure data, we    consider
the following simple elliptic   equation in the whole space $\R^n , n \ge 2$:
\begin{equation}\label{equation in Rn}
  -\Delta u +u  = f \quad\mbox{in}\,\, \R^n .
\end{equation}
Let $u$  be a bounded smooth  solution of (\ref{equation in Rn}) with $f \in C_c^\infty (\R^n )$.
Using the Fourier transform $\mathscr{F} = \widehat{\cdot}$, we obtain
\[
\widehat{u}(\xi ) =  \frac{\widehat{f}(\xi)}{1+ 4 \pi^2  |\xi|^{2}} .
\]
For $s \in \R$, let $ J^s = (I-\Delta )^{s/2}$ be defined via the Fourier transform  by
\[
\widehat{J^s   \varphi}(\xi )= \left(1+ 4 \pi^2  |\xi|^2  \right)^{s/2} \widehat{\varphi} (\xi ).
\]
Then for  $0 \le  \alpha < 1$, we have
\[
\widehat{J^{\alpha +1} u}(\xi ) = \left(1+ 4 \pi^2  |\xi|^2  \right)^{-(1-\alpha)/2}\widehat{f}(\xi)
\]
and
\[
J^{\alpha+1}  u (x)  = J^{-(1-\alpha)} f (x)  = \left( K_{1-\alpha} * f \right) (x) ,
\]
where    $K_{s} (x) = \mathscr{F}^{-1}  \left\{ (1+ 4 \pi^2  |\xi|^2  )^{-s/2}  \right\}(x) $. It is well-known (see, e.g., \cite[Proposition 6.1.5]{Gra2}) that if $0< s< n$,
then
\[
 0<   K_{s} (x) \le  \frac{C(n, s )}{|x|^{n-s}} \quad \mbox{for all}\,\, x \in \R^n \setminus \{ 0\}.
\]
Note   that
$$
K_{1-\alpha} \in L^{\frac{n}{n-1+\alpha},\infty}(\R^n ).
$$
 Hence by Young's convolution inequality (see, e.g., \cite[Theorem 1.2.13]{Gra}), we have
\begin{equation}\label{weak est on Rn}
\left\| (I-\Delta )^{(\alpha+1)/2} \,  u \right\|_{L^{p(\alpha) ,\infty}(\R^n )} \le C(n,\alpha ) \|f\|_{L^1 (\R^n )}
\end{equation}
for every $0 \le  \alpha < 1$, where
\begin{equation}\label{def-pa}
p(\alpha) = \frac{n}{n-1+\alpha}.
\end{equation}
Similarly, if $v$ is a bounded  smooth   solution of
\begin{equation}\label{equation-div in Rn}
-\Delta v + v  = {\rm div}\, G\quad\mbox{in}\,\,\R^n ,
\end{equation}
where $G = (g_1 , \ldots , g_n )$ is a   vector field on $\R^n$, then
\[
\left\| (I-\Delta )^{\alpha/2} v \right\|_{L^{p(\alpha) ,\infty}(\R^n )} \le C(n,\alpha ) \|G\|_{L^1 (\R^n ; \R^n)}
\]
for every $0 \le  \alpha < 1$.

For $-\infty<\alpha<\infty$,   $1< p<\infty$, and $1 \le q \le \infty$, let  $L_\alpha^{p,q}(\R^n )$ be the Banach space of all  tempered  distributions  $\varphi$ on $\R^n $ such that
\[
\|\varphi\|_{L_\alpha^{p,q}(\R^n )} = \|(I-\Delta )^{\alpha/2} \varphi \|_{L^{p,q} (\R^n )}
\]
is finite.  Then by a standard density argument, it   follows from the a priori estimate (\ref{weak est on Rn}) that for each $f \in L^1 (\R^n )$, the problem (\ref{equation in Rn}) has a   solution $u$ satisfying
\begin{equation}\label{weak est on Rn-2}
u\in L_{\alpha+1}^{p(\alpha) ,\infty}(\R^n ) \quad\mbox{and}\quad \left\|   u \right\|_{L_{\alpha+1}^{p(\alpha) ,\infty}(\R^n )} \le C(n,\alpha ) \|f\|_{L^1 (\R^n )}
\end{equation}
for every $0 \le  \alpha < 1$.  It is easily shown that such a solution $u$ is unique. Further regularity of the solution $u$ can be deduced from an embedding result for   Sobolev-Lorentz spaces $L_\alpha^{p,q}(\R^n )$ into classical Besov spaces $B_\alpha^{p,q}(\R^n )$.
Given   $0 \le   \alpha <1$, we choose any $\beta$ such that
 $  \alpha   <  \beta  < 1$. Then since $\beta -\alpha =n/p(\beta) -n/p(\alpha)  $, it follows from the embedding theorem \cite[Theorem 1.2]{seeger}  due to  Seeger and Trebels  (see Lemma \ref{embedding of L into B on domain} below) that $L_{\beta+1}^{p(\beta),\infty}(\R^n )$ is continuously embedded into $B_{\alpha+1}^{p(\alpha) ,\infty}(\R^n )$. Hence the solution $u$ of (\ref{equation in Rn}) also satisfies
\begin{equation}\label{weak est on Rn-3}
u\in B_{\alpha+1}^{p(\alpha) ,\infty}(\R^n ) \quad\mbox{and}\quad \left\|   u \right\|_{B_{\alpha+1}^{p(\alpha) ,\infty}(\R^n )} \le C(n,\alpha ) \|f\|_{L^1 (\R^n )}
\end{equation}
for every  $0 \le  \alpha < 1$. By the same argument, it can be  shown  that
for every  $G \in L^1 (\R^n ; \R^n )$, the problem (\ref{equation-div in Rn}) has a  unique  solution $v$ satisfying
\[
v\in L_{\alpha}^{p(\alpha) ,\infty}(\R^n )\cap B_{\alpha}^{p(\alpha) ,\infty}(\R^n )
\]
and
\begin{equation}\label{weak est on Rn-div-2}
   \left\|  v \right\|_{L_{\alpha}^{p(\alpha) ,\infty}(\R^n )\cap B_{\alpha}^{p(\alpha) ,\infty}(\R^n ) } \le C(n,\alpha ) \|G\|_{L^1 (\R^n ;\R^n )}
\end{equation}
for every $0 \le  \alpha < 1$.

\subsection{Main results}

The main purpose of the paper is to establish  that   solutions of (\ref{eq:D-1-introd}) and  (\ref{eq:D-2-introd}) satisfy optimal   regularity  estimates analogous to (\ref{weak est on Rn-2}), (\ref{weak est on Rn-3}),   and (\ref{weak est on Rn-div-2}).

The data $f$ and  $G$  will be assumed to be only integrable on $\Omega$. We also  allow for even more general $f$ in the Banach space  $\mathcal{M}  (\Omega )$ of all signed  Radon measures   on $\Omega$ with finite total variation. It is well known that $L^1 (\Omega ) \hookrightarrow \mathcal{M}  (\Omega )$, that is,   $L^1 (\Omega )$ is continuously    embedded into $ \mathcal{M}  (\Omega )$.
Furthermore, the boundary data $v_D$ for the problem (\ref{eq:D-2-introd}) can be a signed Radon measure on $\partial \Omega$.

To define weak   solutions of (\ref{eq:D-1-introd})   with $L^1$-  or measure data $f$, we need to introduce Sobolev-Lorentz spaces and Besov spaces on the domain $\Omega$.
For $\alpha>0$,  $1 < p< \infty$, and $1 \le q \le \infty$, we denote by $L_\alpha^{p,q}(\Omega ) $ the space of the restrictions of all functions in $L_\alpha^{p,q}(\R^n)$ to $\Omega$:
$$
L_\alpha^{p,q}(\Omega ) = \left\{ f|_{\Omega}\,:\, f \in L_\alpha^{p,q}(\R^n ) \right\},
$$
which is a Banach space equipped with the usual quotient norm.
For $k \in \N$,   let $W^{k,p,q}(\Omega )$ be the Banach  space of   all $f \in L^{p,q}(\Omega )$ such that $D^\gamma f \in L^{p,q}(\Omega  )$ for all multi-indices $\gamma$ with $|\gamma| \le k$.
Then it can be shown (see Lemma \ref{interp for SL on domains}) that  $L_k^{p,q}(\Omega ) =W^{k,p,q}(\Omega ) $ for $k \in \N$.
If $\alpha >1/p$, each function $f$ in  $L_\alpha^{p,q}(\Omega )$ has a well-defined trace ${\rm Tr} \,f$ in $L^{p,q}(\partial\Omega )$. For $\alpha>0$, $1<p<\infty$, and $1\le q \le \infty$, let  $\mathcal B^{p,q}_{\alpha}(\partial \Omega )$ denote the range space of the trace operator ${\rm Tr}:  L_{\alpha+1/p}^{p,q} (\Omega )  \to L^{p,q}(\partial \Omega)$.
It is well known  (see, e.g.,  \cite[Theorem 3.1]{JK})  that if $1/p< \alpha < 1/p+1$, then $\mathcal B^{p,p}_{\alpha}(\partial \Omega )$ is the boundary Besov space $B_\alpha^{p}(\partial\Omega )$.
 The space of all $f \in L_\alpha^{p,q}(\Omega )$ with ${\rm Tr} \,f =0$ is denoted by $L_{\alpha,0}^{p,q}(\Omega )$.
It can be shown (see Lemma \ref{R-tilde equal zero trac}) that if $1/p< \alpha < 1+1/p$ and $q<\infty$, then $C_c^\infty (\Omega )$ is dense in $L_{\alpha,0}^{p,q}(\Omega )$. For $1<p<\infty$, $1/p< \alpha <1+1/p$, and $1< q \le \infty$, let $L_{-\alpha}^{p,q}(\Omega )$ denote the dual space of $L_{\alpha,0}^{p',q'}(\Omega )$, where $r'=r/(r-1)$ is the H\"{o}lder conjugate to $r$.  Then it will be proved   in  Lemma \ref{embedding of M}  that
\[
\mathcal{M}  (\Omega ) \hookrightarrow L_{-1}^{n/(n-1),\infty}(\Omega ).
\]
Finally, for $\alpha>0$, we define $B_\alpha^{p,q}(\Omega ) = \left\{ f|_{\Omega}\,:\, f \in B_\alpha^{p,q}(\R^n ) \right\}$.

\medskip

For  measure data or more generally $L_{-1}^{n/(n-1),\infty}$-data $f$, weak solutions of (\ref{eq:D-1-introd}) are defined as follows, which is motivated by \cite{KO}.

\begin{dfn}\label{def-solutions-0}
   Suppose that  $f \in L_{-1}^{n/(n-1),\infty}(\Omega )$ and $u_D \in \mathcal B^{n/(n-1),\infty}_{1/n}(\partial \Omega )$.
Then by a   weak solution of (\ref{eq:D-1-introd}), we mean a function $u\in L_{1}^{n/(n-1),\infty}(\Omega )$   such that
\[
 b \cdot \nabla u, \, cu \in L_{loc}^1 (\Omega ) , \quad {\rm Tr}\, u = u_D \,\,\,\mbox{on}\,\,\partial\Omega ,
\]
and
\[
\int_\Omega \left[ \nabla u \cdot \nabla \phi + (  \phi\, b)  \cdot \nabla u +c u\phi \right]   dx =  \langle f ,  \phi \rangle
\quad\mbox{for all}\,\, \phi \in C_c^{\infty} (\Omega ).
\]
\end{dfn}

It should be noticed in Definition \ref{def-solutions-0}  that if $b$ and $c$ satisfy
$$
b \in L^{n,1}(\Omega ; \R^n ) \quad\mbox{and}\quad
 c \in
\begin{cases}
  L^{n/2,1}(\Omega )    & \mbox{if}\,\,n \ge 3 \\
 \bigcup_{s>1} L^s (\Omega ) & \mbox{if}\,\, n=2 ,
 \end{cases}
$$
then by Sobolev's and  H\"{o}lder's   inequalities (see  (\ref{Holder-L}) and Lemma \ref{properties of SL on domains}  below),
$$
b \cdot \nabla u, \, cu \in L^1 (\Omega ) \quad\mbox{for all}\,\, u\in L_{1}^{n/(n-1),\infty}(\Omega ) .
$$

\medskip

We are now ready to state  the first  main result  in the paper.

\begin{thm}\label{Poisson-L1 data}
Let $\Omega$ be a bounded $C^1$-domain in $\R^n , n \ge 2 $. Suppose that $b \in L^{n,1}(\Omega ; \R^n )$,    $ c \in L^{n/2,1}(\Omega )\cap L^s (\Omega )$ for some $s>1$, and $c \ge 0$ in $\Omega$.
Then for every   $f \in L_{-1}^{n/(n-1),\infty}(\Omega )$ and $u_D \in \mathcal B^{n/(n-1),\infty}_{1/n}(\partial \Omega )$, there exists a unique weak  solution  $u$  of (\ref{eq:D-1-introd}), which satisfies
\[
   \|   u \|_{L_{1}^{n/(n-1),\infty}(\Omega )}  \le C \Big( \|f \|_{L_{-1}^{n/(n-1),\infty}(\Omega )} + \|u_D \|_{\mathcal B^{n/(n-1),\infty}_{1/n}(\partial \Omega )} \Big)
\]
for some $C=C(n,  s, \Omega ,b,c )$. In addition, if $ f\in \mathcal{M}  (\Omega )$ and $u_D \in \mathcal B^{p(\alpha ),\infty}_{\alpha +1 -1/p (\alpha )}(\partial \Omega )$ for some  $0 <  \alpha <1$, then
\[
u \in L_{\alpha +1}^{p(\alpha) ,\infty}(\Omega  )  \quad\mbox{and}
\quad
 \|   u \|_{L_{\alpha +1}^{p(\alpha) ,\infty}(\Omega ) } \le C \Big( \|f \|_{\mathcal{M}  (\Omega )}+ \|u_D \|_{\mathcal B^{p(\alpha ),\infty}_{\alpha +1 -1/p (\alpha )}(\partial \Omega )} \Big),
\]
where   $C=C(n,\alpha ,    s, \Omega ,b,c )$. Here, $p(\alpha)=n/(n-1+\alpha)$ as defined in (\ref{def-pa}).
\end{thm}

\begin{rmk}\label{rmk-main1} Let $u$  be the weak solution of (\ref{eq:D-1-introd}) with $ f\in \mathcal{M}  (\Omega )$ and $u_D \in \mathcal B^{p(\alpha ),\infty}_{\alpha +1 -1/p (\alpha )}(\partial \Omega )$ for some  $0 <  \alpha <1$. Then it follows from the embedding result (\ref{embedding of L into L and B})  below    that
\[
u \in   B_{\beta +1}^{p(\beta) ,\infty}(\Omega )
 \quad\mbox{and}\quad
 \|   u \|_{  B_{\beta +1}^{p(\beta) ,\infty}(\Omega  )} \le C \|   u \|_{L_{\alpha +1}^{p(\alpha) ,\infty}(\Omega ) }
\]
for all   $0\le \beta < \alpha$,   where   $C=C(n , \Omega ,\alpha , \beta )$.

Suppose in addition that $u_D \in \mathcal B^{p,p}_{2 -1/p}(\partial \Omega )  $   for some $p>1$.
If  $0 \le  \beta <1$, then  by    Lemma \ref{properties of SL on domains} (i) and (ii), we have
$$
L_2^{p,p}(\Omega ) \hookrightarrow L_{\beta +1}^{p(\beta) ,\infty}(\Omega  ) \quad\mbox{and so}\quad u_D \in \mathcal B^{p,p}_{2 -1/p}(\partial \Omega ) \hookrightarrow \mathcal B^{p(\beta ),\infty}_{\beta +1 -1/p (\beta )}(\partial \Omega ).
$$
Hence it follows from Theorem   \ref{Poisson-L1 data} and (\ref{embedding of L into L and B}) that
\[
u \in L_{\beta +1}^{p(\beta) ,\infty}(\Omega  ) \cap B_{\beta +1}^{p(\beta) ,\infty}(\Omega )
\quad\mbox{for every}\,\, 0 \le \beta < 1.
\]
\end{rmk}

For the  special case when $f \in L^1 (\Omega )$ and $u_D =0$, the existence and the uniqueness assertions in Theorem   \ref{Poisson-L1 data} were already proved    in \cite[Corollary 1.1]{KO}. However the optimal regularity results in Theorem   \ref{Poisson-L1 data} and Remark \ref{rmk-main1} are  completely  new even for   the Poisson equation, which corresponds to the simplest case when $b =0$ and   $c=0$.

\vspace{0.3cm}

The second main result is concerned with optimal regularity properties of  very weak solutions  of (\ref{eq:D-2-introd}) with $L^1$- or measure data.

We begin with   a trace result for very weak solutions of (\ref{eq:D-2-introd}). For $\alpha>0$, $1<p<\infty$, and $1 < q \le \infty$, we denote by  $\mathcal{B}_{-\alpha}^{p,q}(\partial\Omega )$    the dual space of $\mathcal{B}_{\alpha}^{p' ,q'}(\partial\Omega )$.
Indeed, it easily follows from Lemma~\ref{interp for SL on domains}    that if $C^\infty (\partial\Omega ) = \left\{  u|_{\partial\Omega}\,:\, u \in C_c^\infty (\R^n ) \right\}$, then $C^\infty (\partial\Omega )$ is dense in  $\mathcal B^{p',q'}_{\alpha}(\partial \Omega )$ for   $ 1< q \le \infty$.
Hence   $\mathcal{B}_{-\alpha}^{p,q}(\partial\Omega )$ is naturally embedded into the space of distributions on $\partial\Omega$.

Let $\nu$ be  the outward unit normal vector on $\partial\Omega$. The normal derivative of a function $\psi$ on $\partial\Omega$ is  denoted by $ \partial_\nu \psi$, that is,  $\partial_\nu \psi =\nabla \psi \cdot \nu $.

\begin{thm}\label{trace results-VWS} Let  $\Omega$ be a  bounded $C^{1,1}$-domain in $\R^n , n \ge 2$.  Suppose that $b \in L^{n,1}(\Omega ; \R^n )$ and    $ c \in L^{n,1}(\Omega )$.   Suppose in addition  that
\[
  v \in L^{n/(n-1),\infty}(\Omega ), \quad  G \in L^1 (\Omega ; \R^n ),
\]
and
\[
\int_\Omega v \left(- \Delta \psi  + b \cdot \nabla \psi +c\psi \right)  dx = - \int_\Omega G \cdot \nabla \psi \, d x \quad\mbox{for all}\,\, \psi \in C_c^\infty (\Omega ).
\]
Then there exists a unique  $\gamma_0 v \in \mathcal{B}_{-1+1/n}^{n/(n-1),\infty}(\partial\Omega )$ such that
\begin{equation}\label{Green identity}
\left\langle \gamma_0 v ,  \partial_\nu \psi   \right\rangle =\int_\Omega v \left(- \Delta \psi  + b \cdot \nabla \psi +c\psi \right)  dx + \int_\Omega G \cdot \nabla \psi \, d x
\end{equation}
for all $\psi \in C^{1,1}(\overline{\Omega})$ with $\psi =0$ on $\partial\Omega$. Moreover,
\[
\|\gamma_0 v \|_{\mathcal{B}_{-1+1/n}^{n/(n-1),\infty}(\partial\Omega )} \le C \left( \|v\|_{L^{n/(n-1),\infty}(\Omega )}+ \|G\|_{ L^1 (\Omega ; \R^n )} \right)
\]
for some $C=C(n,\Omega )$. In addition, if $v \in L_1^{p}(\Omega )$ for some $p>1$, then  $\gamma_0 v = {\rm Tr\, }v$.
\end{thm}

In particular, Theorem \ref{trace results-VWS} shows that if $v \in L^{n/(n-1),\infty}(\Omega )$ is a distributional solution of $-\Delta v ={\rm div}\, G$ in $\Omega$ for some $G\in L^1 (\Omega ; \R^n )$, then it has a well-defined trace $\gamma_0 v$ belonging to  $\mathcal{B}_{-1+1/n}^{n/(n-1),\infty}(\partial\Omega )$, which extends the trace results in \cite[Theorem 8]{Kim09} and \cite[Lemma 2]{AR11}.

Very weak solutions of (\ref{eq:D-2-introd}) can be defined as follows.

\begin{dfn}\label{def-solutions}  Suppose that $ G \in L^1  (\Omega ;\R^n )$ and $v_D \in \mathcal{B}_{-1+1/n}^{n/(n-1),\infty}(\partial\Omega )$.
Then by a very weak solution of (\ref{eq:D-2-introd}), we mean a function $v \in L^{n/(n-1) ,\infty}(\Omega )$ such that
\[
v|b|,\, cv \in L^1 (\Omega )
\]
and
\[
\left\langle  v_D ,  \partial_\nu \psi   \right\rangle =\int_\Omega v \left(- \Delta \psi  + b \cdot \nabla \psi +c\psi \right)  dx + \int_\Omega G \cdot \nabla \psi \, d x
\]
for all $\psi \in C^{1,1} (\overline{\Omega} )$ with $\psi =0$ on $\partial\Omega$.
\end{dfn}

It follows from Theorem \ref{trace results-VWS}   that if $v$ is a very weak solution    of  (\ref{eq:D-2-introd}) in the sense of Definition \ref{def-solutions}, then it has a well-defined  trace $\gamma_0 v $ in $\mathcal{B}_{-1+1/n}^{n/(n-1),\infty}(\partial\Omega )$ and  $\gamma_0 v = v_D$.

Note also  that if $0 \le \beta <  \alpha < 1$, then
$$
 \mathcal{B}_{\alpha -1/p(\alpha )}^{p(\alpha ),\infty}(\partial\Omega )  \hookrightarrow
\mathcal{B}_{\beta -1/p(\beta )}^{p(\beta ),\infty}(\partial\Omega )  \hookrightarrow \mathcal{B}_{-1+1/n}^{n/(n-1),\infty}(\partial\Omega )
$$
(see the proof of Theorem \ref{Poisson-div-L1 data-introd}).  The following is the second main result in the paper.

\begin{thm}\label{Poisson-div-L1 data-introd}
Let $\Omega$ be a bounded $C^{1,1}$-domain in $\R^n , n \ge 2 $.  Suppose that $b \in L^{n,1}(\Omega ; \R^n )$,    $ c \in L^{n,1}(\Omega )$, and $c \ge 0$ in $\Omega$.
Then for every $ G \in L^1  (\Omega ;\R^n )$  and $v_D \in \mathcal{B}_{-1+1/n}^{n/(n-1),\infty}(\partial\Omega ) $, there exists a unique very weak solution  $v  $  of  (\ref{eq:D-2-introd}), which satisfies
\[
 \|   v \|_{L^{n/(n-1),\infty}(\Omega ) }  \le C \Big( \|G \|_{L^1  (\Omega ;\R^n )} + \|v_D \|_{\mathcal{B}_{-1+1/n}^{n/(n-1),\infty}(\partial\Omega )} \Big)
\]
for some  $C=C(n,  \Omega, b,c )$. In addition, if
$v_D \in \mathcal{B}_{\alpha  - 1/p(\alpha)}^{p(\alpha ),\infty}(\partial\Omega )$ for some $0 <\alpha < 1$, then  we have
\[
v \in L_{\alpha}^{p(\alpha) ,\infty}(\Omega ) \quad\mbox{and}\quad
 \|   v \|_{L_{\alpha}^{p(\alpha) ,\infty}(\Omega )  } \le C \Big( \|G \|_{L^1  (\Omega ;\R^n )} + \|v_D \|_{\mathcal{B}_{\alpha  - 1/p(\alpha)}^{p(\alpha ),\infty}(\partial\Omega )} \Big) ,
\]
 where  $C=C(n,\alpha ,   \Omega, b,c )$.
\end{thm}

\begin{rmk}\label{rmk-main2}
Let $v$  be the very weak solution of  (\ref{eq:D-2-introd}) with $ G \in L^1  (\Omega ;\R^n )$  and $v_D \in \mathcal{B}_{\alpha  - 1/p(\alpha)}^{p(\alpha ),\infty}(\partial\Omega )$ for some $0 <\alpha < 1$. Then it follows from the embedding result (\ref{embedding of L into L and B})     that
\[
v \in   B_{\beta}^{p(\beta) ,\infty}(\Omega )
\quad\mbox{and}\quad
 \|   v \|_{ B_{\beta}^{p(\beta) ,\infty}(\Omega )} \le C \|   v \|_{ L_{\alpha}^{p(\alpha) ,\infty}(\Omega )}
\]
for every $0 \le \beta < \alpha$, where  $C=C(n,\alpha , \beta,   \Omega )$.
Let $\mathcal{M}  (\partial \Omega )$ be the Banach space of all signed  Radon measures   on $\partial \Omega$ with finite total variation.
Then since
$$
\mathcal{B}_{1/p(\beta )-\beta}^{p(\beta) ',1}(\partial\Omega ) = {\rm Tr}\, (L_{1-\beta}^{p(\beta )',1}(\Omega )  ) \quad\mbox{and}\quad L_{1-\beta}^{p(\beta )',1}(\Omega )  \hookrightarrow C(\overline{\Omega})
$$
 (see (\ref{properties of SL on domains})), we have
$$
\mathcal{M}(\partial\Omega ) = \left[ C(\partial \Omega )\right]^*  \hookrightarrow \left[\mathcal{B}_{1/p(\beta )-\beta}^{p(\beta) ',1}(\partial\Omega ) \right]^* =  \mathcal{B}_{\beta -1/p(\beta )}^{p(\beta ),\infty}(\partial\Omega )
$$
for every $0 \le \beta < 1$. Therefore, if  $v_D \in \mathcal{M}(\partial\Omega)$, then the solution $v$ satisfies
\[
v \in L_{\beta}^{p(\beta) ,\infty}(\Omega ) \cap B_{\beta}^{p(\beta) ,\infty}(\Omega )
\quad\mbox{for every}\,\, 0 \le \beta < 1.
\]
\end{rmk}

For the special case when $v_D =0$,  existence and uniqueness of a very weak solution  $v  $  of  (\ref{eq:D-2-introd}) were already proved    in \cite[Corollary 1.2]{KO}. However  the optimal regularity results for  $v$  in Theorem   \ref{Poisson-div-L1 data-introd} and Remark \ref{rmk-main2} are   new  even for   the Poisson equation.

\medskip
The rest of the paper is organized as follows. In the preliminary Section 2, we review standard results for Lorentz spaces and then introduce the definition and basic properties of Sobolev-Lorentz spaces on the whole space $\R^n$. Sobolev-Lorentz spaces on domains $\Omega$ are studied in details in Section 3. Here we provide quite complete proofs of such fundamental properties for Sobolev-Lorentz spaces  as   trace, interpolation, density, duality, embedding, and so on. We also introduce classical Besov spaces on domains and boundary Besov spaces. Section 4 is devoted to studying the Poisson equation  with measure data. In Section 5, making essential use of the results in Section 4, we prove  Theorems \ref{Poisson-L1 data} and \ref{Poisson-div-L1 data-introd} which are the    main optimal regularity results in the paper. In the final Section 6, we provide a detailed proof of   Theorem \ref{trace results-VWS} which is our   trace result for very weak solutions of (\ref{eq:D-2-introd}).

\section{Preliminaries}

For two nonnegative quantities $A$ and $B$, we write $A \lesssim B$ if there is a constant $C$ such that $A \le C B$. If $A \lesssim B$ and $B \lesssim A$, we write $A \approx B$.
For two Banach spaces $X$ and $Y$, we say that $X$ is continuously embedded into $Y$ and write $X \hookrightarrow Y$ if    $X \subset Y$ and  $\| u \|_Y \lesssim \|u\|_{X}$ for all $u \in X$. If $X \hookrightarrow Y$ and $Y \hookrightarrow X$, we write $X=Y$.

\subsection{Lorentz spaces}

Let $\Omega$ be any domain in $\R^n$. For    $1<  p<\infty$  and $1 \le q \le \infty$ or for $1\le p=q \le \infty$, let $ L^{p,q} (\Omega )$ denote the standard Lorentz space on $\Omega$ which is a Banach space equipped with norm $\|\cdot \|_{L^{p,q} (\Omega )}$. Note that $ L^{p,p} (\Omega  )$ coincides with the usual Lebesgue space $L^p (\Omega )$.

H\"{o}lder's inequality can be extended to Lorentz spaces. More precisely, for $1< p,p_1,p_2<\infty$ and  $1\le q,q_1,q_2 \le \infty$ satisfying  $1/p =1/p_1 +1/p_2$ and $1/q \le 1/q_1 +1/q_2$, there exists a constant $C=C(p_1,p_2,q,q_1,q_2)$ such that
\begin{equation}\label{Holder-L}
\|fg\|_{L^{p,q}(\Omega )}\le C \|f\|_{L^{p_1,q_1}(\Omega )} \|g\|_{L^{p_2,q_2}(\Omega)}
\end{equation}
for all $f \in L^{p_1,q_1}(\Omega )$ and $g \in L^{p_2, q_2}(\Omega )$ (see, e.g., \cite[Lemma 2.3]{KO}).

An important property of Lorentz spaces is that a Lorentz space is the real interpolation space of two Lebesgue spaces (see, e.g., \cite[Theorem 5.3.1]{bergh}):    if $ 1 \le   p_0 \neq p_1 \le  \infty $, $0<\theta<1$,   $1/p =(1-\theta)/p_0 + \theta /p_1$, and $1 \le q \le \infty$,
then
\begin{equation}\label{interpolation for L}
\left( L^{p_0 }(\Omega ), L^{p_1}(\Omega ) \right)_{\theta,q} = L^{p,q}(\Omega ).
\end{equation}
Further properties of Lorentz spaces may be deduced from    general results in  real interpolation theory.
Recall that a pair $(X_0 , X_1 )$ of Banach spaces is   compatible if both $X_0$ and $X_1$  are continuously embedded into a common Banach space. If $(X_0 , X_1 )$ is a compatible pair of Banach spaces, then the intersection $X_0 \cap X_1$ and the sum $X_0 +X_1$ are well-defined Banach spaces with the norms
\[
\|x\|_{X_0 \cap X_1}= \max \left\{\|x\|_{X_0} , \|x\|_{X_1} \right\}
\]
and
\[
\|x\|_{X_0 + X_1}= \inf \left\{ \|x_0 \|_{X_0}+ \|x_1\|_{X_1} \,:\, x=x_0 +x_1 ,\,\, x_0 \in X_0 ,\,\,x_1 \in X_1 \right\},
\]
respectively. For a Banach space $X$, we denote by  $X^*$  the dual space of $X$. The dual pairing of   $X$ and   $X^*$ will be  denoted by  $ \langle \cdot , \cdot \rangle_{X^* ,X}$ or simply $ \langle \cdot , \cdot \rangle$.

The following is an immediate consequence of  standard  density and duality results in real interpolation theory (see, e.g., \cite[Sections 3.4 and 3.7]{bergh}).

\begin{lem}\label{density and duality in real interp} Suppose that $(X_0 , X_1 )$ is compatible,  $D$ is a dense subset of $X_0 \cap X_1$, $0< \theta < 1$, and $1 \le q \le \infty$.
\begin{enumerate}[{\upshape (i)}]
\item  If $q<\infty$, then $D$ is dense in $\left( X_0 , X_1 \right)_{\theta , q}$.
\item  Suppose in addition that $D$ is dense in both $X_0$ and $X_1$. If $\widetilde{X}_{\theta, q}$ denotes the closure of $D$ in $\left( X_0 , X_1 \right)_{\theta , q}$, then $[\widetilde{X}_{\theta, q}]^* = \left( X_0^* , X_1^* \right)_{\theta , q'}$. In particular, if $  q < \infty$, then $     [\left( X_0 , X_1 \right)_{\theta , q}]^*=\left( X_0^* , X_1^* \right)_{\theta , q'}$.
\end{enumerate}
\end{lem}

Let $C_c^\infty (\Omega )$ be the space of all functions in $C^\infty (\Omega )$ with support in $\Omega$. Then since $C_c^\infty (\Omega )$ is dense in any of   $L^{p_0}(\Omega ) $,   $  L^{p_1}(\Omega )$, and   $L^{p_0}(\Omega )\cap L^{p_1}(\Omega )$ for $1 \le p_0 ,p_1 <\infty$, it immediately follows from (\ref{interpolation for L}) and Lemma \ref{density and duality in real interp} that    if $\widetilde{L}^{p,q}(\Omega )$ is the closure of $C_c^\infty (\Omega )$ in $L^{p,q}(\Omega )$, then
\begin{equation}\label{widehat-L}
\left[ \widetilde{L}^{p,q} (\Omega )\right]^*
= L^{p',q'} (\Omega )  \quad\mbox{for}\,\, 1<p<\infty, \,\,1 \le q \le \infty.
\end{equation}
Moreover, if $1 \le q  <  \infty$, then $C_c^\infty (\Omega )$ is dense in $L^{p,q} (\Omega )$ and
\begin{equation}\label{duality for L}
\left[ L^{p,q} (\Omega )\right]^*
= L^{p',q'} (\Omega ) .
\end{equation}
Hence  $L^{p,q} (\Omega )$ is reflexive for $1 <p,q<\infty$.

It is well known that the family of  Lebesgue spaces  $L^p (\Omega )$ is a  complex interpolation scale for $1 \le p \le \infty$: more precisely,
if $1 \le   p_0 , p_1 \le \infty$,  $0 < \theta <  1$,  and  $1/p=(1-\theta)/p_0 +\theta/p_1$, then
\[
\left[\,  L^{p_0}(\Omega ) , L^{p_1}(\Omega ) \, \right]_{\theta} = L^{p}(\Omega ).
\]
The following  is a  complex counterpart of Lemma \ref{density and duality in real interp} (see  \cite[Sections 4.2 and 4.5]{bergh}).

\begin{lem}\label{density and duality in complex interp} Suppose that $(X_0 , X_1 )$ is compatible, $D$ is a dense subset of $X_0\cap X_1$,  and    $0< \theta < 1$.
\begin{enumerate}[{\upshape (i)}]
\item   $D$ is dense in $\left[ X_0 , X_1 \right]_{\theta}$.
\item  Suppose in addition that $X_0$ or $X_1$ is reflexive and $D$ is dense in both $X_0$ and $X_1$. Then $      \left[ X_0 , X_1 \right]_{\theta}^*=\left[ X_0^* , X_1^* \right]_{\theta}$.
\end{enumerate}
\end{lem}

\subsection{Sobolev-Lorentz spaces on $\R^n$}

As in Subsection 1.1, for $-\infty<\alpha<\infty$,   $1< p<\infty$, and $1 \le q \le \infty$, we denote by $L_\alpha^{p,q}(\R^n )$ the space of all  tempered  distributions  $f$ on $\R^n $ such that $(I-\Delta )^{\alpha/2} f   \in L^{p,q} (\R^n )$. The Sobolev-Lorentz space $L_\alpha^{p,q}(\R^n )$ is a Banach space equipped with the   norm
\[
\|f\|_{L_\alpha^{p,q}(\R^n )} = \|(I-\Delta )^{\alpha/2} f \|_{L^{p,q} (\R^n )}.
\]
Note in particular that $L_0^{p,q}(\R^n )= L^{p,q}(\R^n )$.
If $1<p=q<\infty$, we write $L_\alpha^{p}(\R^n )= L_\alpha^{p,p}(\R^n )$.
Recall then from   interpolation by the complex method  (see, e.g., \cite[Theorem 6.4.5]{bergh}) that if $0<\theta < 1$, $\alpha =(1-\theta )\alpha_0 +\theta \alpha_1$, and  $1/p=(1-\theta)/p_0 +\theta/p_1$, then
\[
 L_\alpha^{p}(\R^n ) = \left[\,  L_{\alpha_0}^{p_0}(\R^n ) , L_{\alpha_1}^{p_1}(\R^n ) \, \right]_{\theta}.
\]
By definition, $(I-\Delta )^{\alpha/2}$ is an isometric isomorphism from
$L_\alpha^{p,q}(\R^n )$ onto $ L^{p,q}(\R^n )$ and its inverse is $(I-\Delta )^{-\alpha/2}$. Hence it immediately follows  from the   interpolation result (\ref{interpolation for L}) for Lorentz spaces  that if $ 1<  p_0 \neq p_1 < \infty $, $0<\theta<1$,    $1/p =(1-\theta)/p_0 + \theta /p_1$, and $1 \le q \le \infty$, then
\begin{equation}\label{interpolation for SL on Rn}
L_\alpha^{p,q}(\R^n )= \left( L_\alpha^{p_0 }(\R^n ), L_\alpha^{p_1}(\R^n ) \right)_{\theta,q}   .
\end{equation}

\medskip
The following are some basic properties of  Sobolev-Lorentz spaces $L_\alpha^{p,q}(\R^n)$.

\begin{lem}\label{properties of SL on Rn} Let $ -\infty<  \alpha, \alpha_0 < \infty$,  $1< p , p_0< \infty$, and    $1 \le q, q_0 \le \infty$.
\begin{enumerate}[{\upshape (i)}]
\item  If $\alpha  \ge \alpha_0$  and $q \le q_0 $, then   $L_\alpha^{p,q}(\R^n ) \hookrightarrow L_{\alpha_0}^{p,q_0 }(\R^n )  $. In particular, if $\alpha > 0$, then $L_\alpha^{p,q}(\R^n ) \hookrightarrow L^{p,q  }(\R^n )\hookrightarrow L_{-\alpha}^{p,q}(\R^n )$.
\item  If $\alpha-\alpha_0 = n/p-n/p_0 >0$, then $L_\alpha^{p,q}(\R^n)\hookrightarrow L_{\alpha_0}^{p_0 , q}(\R^n)$. In particular, if $0<  \alpha p < n $, then $L_\alpha^{p,q}(\R^n)\hookrightarrow L^{np/(n-\alpha p),q}(\R^n)$.
\item   If $\alpha p=n$, then $L_\alpha^{p,1}(\R^n)\hookrightarrow C_0 (\R^n )$, where $C_0 (\R^n )$ is the closure of $C_c^\infty (\R^n)$ in $L^\infty (\R^n )$.
\item   If $k$ is a positive integer, then $L_k^{p,q}(\R^n) =W^{k,p,q}(\R^n )$, where $W^{k,p,q}(\R^n )$ is the Banach space  of  all $f \in L^{p,q}(\R^n )$ such that $D^\gamma f \in L^{p,q}(\R^n )$ for all multi-indices $\gamma$ with $|\gamma | \le k$.
\item  If $q < \infty$, then $C_c^\infty (\R^n)$ is dense in  $L_\alpha^{p,q}(\R^n)$ and $\left[L_\alpha^{p,q}(\R^n) \right]^* = L_{-\alpha}^{p',q'}(\R^n)$.
\end{enumerate}
\end{lem}

\begin{proof}
Parts (i)-(v)  are essentially known.
For instance, the embedding results (i) and (ii) are just special cases of \cite[Theorem 5.2]{BKO} and Part (iv) was proved in \cite[Theorem  5.5]{BKO}. 

To prove Part (v), suppose that $q<\infty$. Choose  $k\in \N$ and  $ 1<  p_0 \neq p_1 < \infty $ such that  $k>\alpha$ and  $2/p =1/p_0 + 1 /p_1$. Then by (iv), \eqref{interpolation for SL on Rn},  and (i),  we have
$$
\left( W^{k, p_0 }(\R^n ), W^{k, p_1}(\R^n ) \right)_{1/2,q} = L_k^{p,q}(\R^n )\hookrightarrow L_{\alpha}^{p,q }(\R^n ).
$$
It is well known that $C^\infty_c(\R^n )$ is dense in $W^{k, p_0 }(\R^n )\cap W^{k, p_1}(\R^n )$. Hence it follows from Lemma~\ref{density and duality in real interp} (i) that   $C^\infty_c(\R^n )$ is dense in $L_k^{p,q}(\R^n )$.
To prove that  $C_c^\infty (\R^n)$ is dense in  $L_\alpha^{p,q}(\R^n)$, it thus remains to prove that $L_k^{p,q}(\R^n )$ is dense in  $L_{\alpha}^{p,q }(\R^n )$. Let $f\in L_{\alpha}^{p,q }(\R^n )$ be given. Then by definition,  $g =(I-\Delta )^{\alpha/2}f$  belongs to $L^{p,q}(\R^n )$. Since $q<\infty$,  $C_c^\infty (\R^n)$ is dense in  $L^{p,q}(\R^n)$. Hence there is a sequence  $\{g_j \}$  in $C_c^\infty (\R^n)$ such that $g_j \to g$ in $L^{p,q}(\R^n)$. For each $j$, define $f_j = (I-\Delta )^{-\alpha/2}g_j$. Then   $\{ f_j  \}$ is a sequence in $L_k^{p,q}(\R^n )$ that converges to  $ f$ in $L_{\alpha}^{p,q }(\R^n )$.  This proves  the density result in Part (v).
Then the duality result in Part (v)  easily follows from the duality result (\ref{duality for L}) for Lorentz spaces since $(I-\Delta )^{\alpha/2}$ is an isometric isomorphism from $L_\alpha^{p,q}(\R^n )$ onto $ L^{p,q}(\R^n )$.

We finally  prove Part (iii). By (v), it suffices to prove that
\begin{equation}\label{embedding into BC}
\|f\|_{L^\infty (\R^n )} \le C(n,p) \|f\|_{L_{n/p}^{p,1 }(\R^n )} \quad \mbox{for all}\,\, f \in C^\infty_c(\R^n ).
\end{equation}
Although (\ref{embedding into BC}) is a consequence of \cite[Theorem  5.8 (iii)]{BKO}, here we provide a rather simple  proof for the sake of the readers' convenience.
Let $f \in C_c^\infty (\R^n)$  and $g = (I-\Delta )^{\alpha} f$, where $\alpha=n/p$. Then
\[
f (x)= (K_\alpha * g )(x) =\int_{\R^n} K_\alpha (y) g(x-y)\, dy
\]
for some  $K_\alpha  \in L^{p',\infty} (\R^n )$ as   in Subsection \ref{sec1.1}. Hence by H\"older's inequality (\ref{Holder-L}) in Lorentz spaces, we have
\begin{align*}
|f(x)| \lesssim \|K_{\alpha} \|_{L^{p',\infty} (\R^n )} \|g(x-\cdot ) \|_{L^{p,1} (\R^n )}  \lesssim   \|g  \|_{L^{p,1} (\R^n )} =   \|f \|_{L_{n/p}^{p,1}(\R^n)}
\end{align*}
for   all  $x\in \R^n$, which proves  (\ref{embedding into BC}).
 This completes the proof of the lemma.
\end{proof}

\section{Sobolev-Lorentz spaces on  domains}

Throughout this section, let $\Omega$ be a  bounded Lipschitz domain in $\R^n$. If $f$ is a function on $\R^n$, let $R_\Omega f$ denote the restriction of $f$ to $\Omega$.

\subsection{Definitions and basic properties}

As in Subsection 1.2, for $\alpha  \ge  0$, $1< p<\infty$, and $1 \le q \le \infty$, let
$L_\alpha^{p,q}(\Omega )$ be the space of the   restrictions to $\Omega$ of all functions in $L_\alpha^{p,q}(\R^n )$,
which is a Banach space equipped with the usual quotient  norm
\[
\|u\|_{L_\alpha^{p,q}(\Omega )} = \inf \left\{ \|f\|_{L_\alpha^{p,q}(\R^n )}  \,:\, f\in L_\alpha^{p,q}(\R^n ),\ R_\Omega f = u \right\}.
\]
By definition,  $R_\Omega$ is bounded   from $L_\alpha^{p,q}(\R^n )$ into $L_\alpha^{p,q}(\Omega )$. Note also  that    $L_0^{p,q}(\Omega ) = L^{p,q}(\Omega )$ and    $L_0^{p,p} (\Omega )= L^p (\Omega )$. If $1< p=q <\infty$,  we write $L_\alpha^{p }(\Omega )=L_\alpha^{p,p}(\Omega )$. Then it is well known (see, e.g.,  \cite[p.329]{FMM}) that $L_\alpha^{p }(\Omega )$ is   reflexive for $\alpha \ge  0$ and $1<p<\infty$. Moreover,   it follows from     \cite[Proposition  2.4]{JK}   that $L_\alpha^{p }(\Omega )$   is a   complex interpolation scale  for $\alpha \ge  0$ and $1<p<\infty$, that is, if $0 \le      \alpha_0 , \alpha_1 < \infty$,  $1<     p_0 , p_1 < \infty$, $0<\theta < 1$,
$\alpha =(1-\theta )\alpha_0 +\theta\alpha_1$, and $1/p =(1-\theta )/p_0 +\theta/p_1$, then
\begin{equation}\label{complex interp for S}
L_{\alpha}^{p} (\Omega )  = \left[ L_{\alpha_0}^{p_0} (\Omega ) ,L_{\alpha_1}^{p_1} (\Omega ) \right]_{\theta}  .
\end{equation}

For a positive integer $k$, let $W^{k,p,q}(\Omega )$ be the Banach space of   all $f \in L^{p,q}(\Omega )$ such that $D^\gamma f \in L^{p,q}(\Omega  )$ for all multi-indices $\gamma$ with $|\gamma | \le k$. We write $W^{0,p,q}(\Omega)= L^{p,q}(\Omega)$.
If $1< p=q <\infty$,  we write    $W^{k,p}(\Omega ) = W^{k,p,p}(\Omega )$.
Then it is also  well known (see, e.g., \cite[(2.2)]{JK}) that if $k$ is any positive integer, then
\begin{equation}\label{L equals W}
L_k^p (\Omega )= W^{k,p}(\Omega )\quad\mbox{for}\,\,1 <p <\infty.
\end{equation}
As usual, $W^{k,p}(\Omega )$ is defined even for $p=1$ or $p=\infty$. By the universal extension theorem due to Stein \cite[Chapter VI]{stein}, there exists a   linear    operator $E_\Omega : L^{1}(\Omega ) \to L^{1}(\R^n )$ such that
\begin{enumerate}
\item[\textup{(1)}] $R_\Omega E_\Omega u = u$ for all $u\in L^{1}(\Omega)$;
\item[\textup{(2)}] $E_\Omega$ is bounded from $W^{k,p}(\Omega )$ into $W^{k,p}(\R^n )$ for all $k\in \N\cup \{0\}$ and $1 \le p \le \infty$.
\end{enumerate}


\begin{lem}\label{interp for SL on domains} Let $0 \le   \alpha  <\infty$, $1< p <\infty$, and $ 1\le q  \le \infty$.
\begin{enumerate}[{\upshape (i)}]
\item $E_\Omega$ is bounded from $L_\alpha^{p,q}(\Omega )$ into $L_\alpha^{p,q}(\R^n )$.
\item  If      $1< p_0 \neq p_1 < \infty$, $0<\theta < 1$, and $1/p=(1-\theta)/p_0 +\theta/p_1$,    then
    $$
    L_\alpha^{p,q}(\Omega ) = \left(\,  L_\alpha^{p_0}(\Omega ) , L_\alpha^{p_1}(\Omega ) \, \right)_{\theta,q}   .
    $$
\item     If $k$ is a positive integer, then $L_k^{p,q}(\Omega ) =W^{k,p,q}(\Omega )$.
\item  If   $q<\infty$, then  $C^\infty (\overline{\Omega}) $ is dense in $L_\alpha^{p,q}(\Omega )$, where $C^\infty (\overline{\Omega}) =R_\Omega \left(C_c^\infty (\R^n ) \right)$.
\end{enumerate}
\end{lem}

\begin{proof}
Suppose that  $1< p_0 \neq p_1 < \infty$, $0<\theta < 1$, and $1/p=(1-\theta)/p_0 +\theta/p_1$.

By definition,  the restriction operator  $R_\Omega$ is bounded from $L_\alpha^{p_i}(\R^n  )$ into  $L_\alpha^{p_i}(\Omega )$   for each $i=0,1$. Hence, by real interpolation,   it follows from (\ref{interpolation for SL on Rn}) that $R_\Omega$ is bounded from $ L_\alpha^{p,q}(\R^n )$ into $\left(\,  L_\alpha^{p_0}(\Omega ) , L_\alpha^{p_1}(\Omega ) \, \right)_{\theta,q}$. This implies that $ L_\alpha^{p,q}(\Omega  ) = R_\Omega  ( L_\alpha^{p,q}(\R^n ) ) \hookrightarrow \left(\,  L_\alpha^{p_0}(\Omega ) , L_\alpha^{p_1}(\Omega ) \, \right)_{\theta,q}$. On the other hand, by  \cite[Proposition 2.4]{JK},  $E_\Omega$ is bounded from $L_\alpha^{p_i}(\Omega )$ into $L_\alpha^{p_i}(\R^n  )$ for each $i=0,1$. Hence by real interpolation, $E_\Omega$ is bounded from $\left(\,  L_\alpha^{p_0}(\Omega ) , L_\alpha^{p_1}(\Omega ) \, \right)_{\theta,q}$ into $ L_\alpha^{p,q}(\R^n )$.
It is then easy to check that $\left(\,  L_\alpha^{p_0}(\Omega ) , L_\alpha^{p_1}(\Omega ) \, \right)_{\theta,q} \hookrightarrow    L_\alpha^{p,q}(\Omega  )$.  This proves  Part (i) as well  as  Part (ii).

To prove Part (iii), let $k$ be a positive integer. Then it follows from a nontrivial
interpolation result due to DeVore and Scherer   \cite[Theorem 2]{DS} (see also \cite[Lemma 2.6]{KO}) that
\[
\left(W^{k,p_0}(\Omega ), W^{k,p_1}(\Omega ) \right)_{\theta,q} = W^{k,p,q}(\Omega ).
\]
Hence by   Part (ii) and (\ref{L equals W}), we have
\[
L_k^{p,q}(\Omega ) = \left(W^{k,p_0}(\Omega ), W^{k,p_1}(\Omega ) \right)_{\theta,q} = W^{k,p,q}(\Omega ),
\]
which proves Part (iii).
Finally, Part  (iv) is easily deduced from  Lemma \ref{properties of SL on Rn}  (v).
\end{proof}

For $\alpha  \ge  0$, $1< p<\infty$, and $1 \le q \le \infty$,
let $\widetilde{L}_\alpha^{p,q} (\overline{\Omega} )$ be the  space of  all $f \in L_\alpha^{p,q}(\R^n )$ with support in $\overline{\Omega}$, which is a closed subspace of $L_\alpha^{p,q}(\R^n )$.   Then  $R_\Omega$ is an injection from $\widetilde{L}_\alpha^{p,q} (\overline{\Omega} )$ into $L_\alpha^{p,q}(\Omega  )$. Hence $R_\Omega  (\widetilde{L}_\alpha^{p,q} (\overline{\Omega} )  )$, the image of $\widetilde{L}_\alpha^{p,q} (\overline{\Omega} )$ under $R_\Omega$,  is a Banach space equipped with the norm
\[
\|R_\Omega f  \|_{R_\Omega  (\widetilde{L}_\alpha^{p,q} (\overline{\Omega} )  )} =  \|f\|_{L_\alpha^{p,q}(\R^n )} \quad \mbox{for all}\,\,  f\in \widetilde{L}_\alpha^{p,q} (\overline{\Omega} ) .
\]
For a function  $u$ on $\Omega$, we denote by $\widetilde{u}$ the zero extension of $u$ to $\R^n$, defined by
\[
\widetilde{u} (x)=
\begin{cases}
  u (x)    & \mbox{for}\,\,x \in \Omega  \\
 \,\,\, 0   & \mbox{for}\,\,x \in \R^n \setminus \Omega.
 \end{cases}
\]
Then since $\partial\Omega$ has measure zero, it immediately follows  that
$$
  R_\Omega  (\widetilde{L}_\alpha^{p,q} (\overline{\Omega} ) ) = \left\{ u \in L^{p,q}(\Omega ) \, :\, \widetilde{u} \in L_\alpha^{p,q}(\R^n ) \right\} \hookrightarrow L_\alpha^{p,q}(\Omega )
$$
and
\[
\|u\|_{R_\Omega  (\widetilde{L}_\alpha^{p,q} (\overline{\Omega} )  )} =  \|\widetilde{u}\|_{L_\alpha^{p,q}(\R^n )} \ge \|u\|_{L_\alpha^{p,q}(\Omega )} \quad\mbox{for all}\,\, u \in R_\Omega  (\widetilde{L}_\alpha^{p,q} (\overline{\Omega} ) ).
\]
Let    $\widetilde{L}_\alpha^{p,q}(\Omega )$ denote  the closure of $C_c^\infty (\Omega  )$ in $R_\Omega  (\widetilde{L}_\alpha^{p,q} (\overline{\Omega} )  )$.
Note  that
$$
R_\Omega  (\widetilde{L}_0^{p,q} (\overline{\Omega} ) ) = L^{p,q}(\Omega ),  \quad \widetilde{L}_0^{p,q}(\Omega ) =  \widetilde{L}^{p,q}(\Omega ),
\quad\mbox{and}\quad
\left[\widetilde{L}_0^{p,q}(\Omega )\right]^* = L^{p',q'}(\Omega ) .
$$
For $\alpha >0$, we define $L_{-\alpha}^{p,q}(\Omega )$ as the dual space of $\widetilde{L}_{\alpha}^{p',q'}(\Omega )$.
Finally, if $ 1< p=q  <\infty$,  we write    $\widetilde{L}_\alpha^{p}(\overline{\Omega} ) = \widetilde{L}_\alpha^{p,p}(\overline{\Omega} )$, $\widetilde{L}_\alpha^{p}(\Omega ) = \widetilde{L}_\alpha^{p,p}(\Omega )$,  and  $L_{-\alpha}^{p}(\Omega ) = L_{-\alpha}^{p,p}(\Omega )$.
Then it was shown in  \cite[Remark 2.7]{JK} and   \cite[p.329]{FMM} (see also \cite[Theorem 3.5]{triebel02}) that
$C_c^\infty (\Omega  )$ is dense in $R_\Omega  (\widetilde{L}_\alpha^{p} (\overline{\Omega} )  )$, that is, $\widetilde{L}_\alpha^{p}(\Omega ) = R_\Omega  (\widetilde{L}_\alpha^{p} (\overline{\Omega} )  ) $  and the Banach  spaces  $\widetilde{L}_\alpha^{p}(\Omega )$  and    $L_{-\alpha}^{p}(\Omega )$ are reflexive. Moreover it follows from   \cite[Proposition  2.11]{JK}  and Lemma \ref{density and duality in complex interp} that $\widetilde{L}_\alpha^{p}(\Omega )$  and    $L_{-\alpha}^{p}(\Omega )$ are complex interpolation scales for $\alpha \ge 0$ and $1<p<\infty$: if $0 \le     \alpha_0 , \alpha_1 < \infty$,  $1<    p_0 , p_1 < \infty$, $0<\theta < 1$, $\alpha =(1-\theta )\alpha_0 +\theta\alpha_1$, and $1/p =(1-\theta )/p_0 +\theta/p_1$, then
\begin{equation}\label{comp-interp-S on domains}
\widetilde{L}_{\alpha}^{p} (\Omega )  = \left[\widetilde{L}_{\alpha_0}^{p_0} (\Omega ) ,\widetilde{L}_{\alpha_1}^{p_1} (\Omega ) \right]_{\theta} \quad\mbox{and}\quad L_{-\alpha}^{p} (\Omega )  =   \left[ L_{-\alpha_0}^{p_0} (\Omega ) ,L_{-\alpha_1}^{p_1} (\Omega ) \right]_{\theta}.
\end{equation}


\begin{lem}\label{real-interp-S on domains} Let $0 \le   \alpha  <\infty$, $1< p  < \infty$,  and $ 1\le q  \le \infty$.
\begin{enumerate}[{\upshape (i)}]
\item  If      $1< p_0 \neq p_1 < \infty$, $0<\theta < 1$, and $1/p=(1-\theta)/p_0 +\theta/p_1$,    then
    \[
 R_\Omega  (\widetilde{L}_\alpha^{p,q} (\overline{\Omega} ) )  = \left( \widetilde{L}_\alpha^{p_0}(\Omega ) ,\widetilde{L}_\alpha^{p_1}(\Omega ) \right)_{\theta,q}
\quad\mbox{and}\quad
     L_{-\alpha}^{p,q} (\Omega )   = \left( L_{-\alpha}^{p_0} (\Omega ) ,L_{-\alpha}^{p_1} (\Omega ) \right)_{\theta,q}.
\]
\item If   $q<\infty$, then  $C_c^\infty (\Omega) $ is dense in $R_\Omega  (\widetilde{L}_\alpha^{p,q} (\overline{\Omega} ) )$, that is,   $\widetilde{L}_\alpha^{p,q}(\Omega )= R_\Omega  (\widetilde{L}_\alpha^{p,q} (\overline{\Omega} ) )$.
\item If  $k$ is  a positive integer, then $\widetilde{L}_k^{p,q}(\Omega ) $ is the closure of $C_c^\infty (\Omega) $ in $L_k^{p,q}(\Omega )$.
\item If $q>1$, then $C_c^\infty (\Omega )$ is dense in $L_{-\alpha}^{p',q'} (\Omega ) $ and $ \left[ L_{-\alpha}^{p', q'}(\Omega ) \right]^* = R_\Omega  (\widetilde{L}_\alpha^{p,q} (\overline{\Omega} ) )$.
\end{enumerate}
\end{lem}

\begin{proof}
Suppose that    $1< p_0 \neq p_1 < \infty$, $0<\theta < 1$, and $1/p=(1-\theta)/p_0 +\theta/p_1$.

We first prove Part (i) by following Mayboroda \cite{Mayb05}. Consider the operator
\[
P  = I  - E_{\R^n \setminus \overline{\Omega}}  \, R_{\R^n \setminus \overline{\Omega}} : L_\alpha^{p_0}(\R^n )+ L_\alpha^{p_1}(\R^n ) \rightarrow L_\alpha^{p_0}(\R^n )+ L_\alpha^{p_1}(\R^n ),
\]
where $R_{\R^n \setminus \overline{\Omega}}$ is the restriction operator to $\R^n \setminus \overline{\Omega}$  and $E_{\R^n \setminus \overline{\Omega}}$ is   Stein's universal extension operator.
Then $P$ is a projection (that is, a bounded linear operator with $P^2 =P$) on $L_\alpha^{p_0}(\R^n )$, $L_\alpha^{p_1}(\R^n )$, and $L_\alpha^{p,q}(\R^n )$, respectively.
It is easy to show that
$$
\widetilde{L}_\alpha^{p,q} (\overline{\Omega} )     = P (L_\alpha^{p,q}(\R^n ) ).
$$
Hence by (\ref{interpolation for SL on Rn}) and \cite[Lemma 3.1.9]{Mayb05},
\begin{align*}
\widetilde{L}_\alpha^{p,q} (\overline{\Omega} )     = P (L_\alpha^{p,q}(\R^n ) ) & = P  (\left( L_\alpha^{p_0 }(\R^n ), L_\alpha^{p_1}(\R^n ) \right)_{\theta,q}  ) \\
&=  \left( P  ( L_\alpha^{p_0 }(\R^n )  ), P  ( L_\alpha^{p_1}(\R^n )  ) \right)_{\theta,q} \\
& =\left( \widetilde{L}_\alpha^{p_0}(\overline{\Omega} )   ,\widetilde{L}_\alpha^{p_1}(\overline{\Omega} )   \right)_{\theta,q}.
\end{align*}
Moreover, since $R_\Omega  (\widetilde{L}_\alpha^{p_i } (\overline{\Omega} ) ) = \widetilde{L}_\alpha^{p_i}(\Omega )$ for $i=0,1$, we have
\[
R_\Omega  (\widetilde{L}_\alpha^{p,q} (\overline{\Omega} ) )
 = R_\Omega \left( \widetilde{L}_\alpha^{p_0}(\overline{\Omega} )   ,\widetilde{L}_\alpha^{p_1}(\overline{\Omega} )   \right)_{\theta,q}
  \hookrightarrow \left( \widetilde{L}_\alpha^{p_0}(\Omega )  ,\widetilde{L}_\alpha^{p_1}(\Omega ) \right)_{\theta,q}.
\]
To prove the reverse embedding, let  $E^0_\Omega$ denote the zero extension operator from $\Omega$ to $\R^n$, defined by $E^0_\Omega u = \widetilde{u}$. It immediately follows from  (\ref{interpolation for SL on Rn}) that 
$$
E_\Omega^0  \left( \widetilde{L}_\alpha^{p_0}(\Omega )  ,\widetilde{L}_\alpha^{p_1}(\Omega ) \right)_{\theta,q} \hookrightarrow  \left( L_\alpha^{p_0 } (\R^n ) )  , L_\alpha^{p_1 } ( \R^n )   \right)_{\theta,q} = L_\alpha^{p,q }(\R^n).
$$
Hence  using the fact that $R_\Omega\circ P\circ E^0_\Omega=I$, we have  
$$
\begin{aligned}
 \left( \widetilde{L}_\alpha^{p_0}(\Omega )  ,\widetilde{L}_\alpha^{p_1}(\Omega ) \right)_{\theta,q}
&= \left( R_\Omega \circ P \circ E^0_\Omega \right) \left( \widetilde{L}_\alpha^{p_0}(\Omega )  ,\widetilde{L}_\alpha^{p_1}(\Omega ) \right)_{\theta,q} \\
& \hookrightarrow  ( R_\Omega \circ P) (  L_\alpha^{p,q }(\R^n) ) = R_\Omega  (\widetilde{L}_\alpha^{p,q} (\overline{\Omega} ) ).
\end{aligned}
$$
which proves the first interpolation result in Part (i), that is,  
\[
R_\Omega  (\widetilde{L}_\alpha^{p,q} (\overline{\Omega} ) )  = \left( \widetilde{L}_\alpha^{p_0}(\Omega ) ,\widetilde{L}_\alpha^{p_1}(\Omega ) \right)_{\theta,q}.
\]
Recall  that   $\widetilde{L}_\alpha^{p,q}(\Omega )$ is  the closure of $C_c^\infty (\Omega  )$ in $R_\Omega  (\widetilde{L}_\alpha^{p,q} (\overline{\Omega} )  )$. Hence  $\widetilde{L}_\alpha^{p,q}(\Omega )$ is  the closure of $C_c^\infty (\Omega  )$ in $\left( \widetilde{L}_\alpha^{p_0}(\Omega ) ,\widetilde{L}_\alpha^{p_1}(\Omega ) \right)_{\theta,q}$. Therefore, by the duality theorem (Lemma \ref{density and duality in real interp} (ii)), we have
\[
L_{-\alpha}^{p',q'} (\Omega )   = \left[\widetilde{L}_\alpha^{p,q}(\Omega )\right]^* =  \left( L_{-\alpha}^{p_0 '} (\Omega ) ,L_{-\alpha}^{p_1 '} (\Omega ) \right)_{\theta,q'},
\]
which proves Part (i).   Part (ii) immediately follows from Part (i) and Lemma \ref{density and duality in real interp} (i).

To prove Part (iii), let $k$ be a positive integer. Since $R_\Omega  (\widetilde{L}_k^{p,q} (\overline{\Omega} ) ) \hookrightarrow L_k^{p,q}(\Omega )$, it is obvious that $\widetilde{L}_k^{p,q}(\Omega ) $ is a subset of  the closure of $C_c^\infty (\Omega) $ in $L_k^{p,q}(\Omega )$.
Suppose that $u$ belongs to  the closure of $C_c^\infty (\Omega) $ in $L_k^{p,q}(\Omega )$. Then there exists a sequence $\{u_j\}$ in $C_c^\infty (\Omega )$ such that $u_j \to u$ in $L_k^{p,q}(\Omega )$.
By Lemma \ref{properties of SL on Rn} (iv) and Lemma \ref{interp for SL on domains} (iii), we have
\begin{align*}
\| u_j -u_m  \|_{R_\Omega  (\widetilde{L}_k^{p,q} (\overline{\Omega} )  )}  &=  \|\widetilde{u}_j - \widetilde{u}_m \|_{L_k^{p,q}(\R^n )} \\
 & \approx  \|\widetilde{u}_j - \widetilde{u}_m  \|_{W^{k, p,q}(\R^n )} \\
 &= \| {u}_j -u_m  \|_{W^{k, p,q}(\Omega )}\approx \| {u}_j -u_m  \|_{L_k^{p,q}(\Omega )}
\end{align*}
for all $j,m \ge 1$. Hence  $ \{u_j\} $ converges in $R_\Omega  (\widetilde{L}_k^{p,q} (\overline{\Omega} )  )$ to some $v$. Since  $ u_j \to u $ and $u_j \to v$  in $L_k^{p,q}(\Omega )$, we deduce that
$u = v \in \widetilde{L}_k^{p,q}(\Omega )$, which proves Part (iii).

It remains to prove Part (iv).
Suppose that $1<q \le \infty$. We may assume that $0<\alpha < k$ for some  $k \in \N$.  Then by (\ref{comp-interp-S on domains}), we have
\[
L_{-\alpha}^{p_0 '} (\Omega )  =   \left[ L^{p_0'} (\Omega ) ,L_{-k}^{p_0 '} (\Omega ) \right]_{\alpha /k} .
\]
Since $C_c^\infty (\Omega )$ is dense in $L^{p_0'} (\Omega ) = L^{p_0'} (\Omega ) \cap L_{-k}^{p_0 '} (\Omega )$, it follows from Lemma \ref{density and duality in complex interp} (i) that
$C_c^\infty (\Omega )$ is dense in
 $L_{-\alpha}^{p_0 '} (\Omega )$. Similarly, $C_c^\infty (\Omega )$ is dense in  $   L_{-\alpha}^{p_1 '} (\Omega )$.
Since $\Omega$ is bounded, $C_c^\infty (\Omega )$ is also dense in  $ L_{-\alpha}^{p_0 '} (\Omega )\cap   L_{-\alpha}^{p_1 '} (\Omega )$.   Recall that $\widetilde{L}_\alpha^{p_i}(\Omega )$ is reflexive and its dual is    $L_{-\alpha}^{p_i'}(\Omega )$ for $i=0,1$.
Hence it follows from Part (i)  and   Lemma \ref{density and duality in real interp} that $C_c^\infty (\Omega )$ is dense in $L_{-\alpha}^{p',q'} (\Omega ) $ and
\[
 \left[ L_{-\alpha}^{p', q'}(\Omega ) \right]^* =   \left(\,    \widetilde{L}_{\alpha}^{p_0 }(\Omega )   , \ \widetilde{L}_{\alpha}^{p_1}(\Omega )   \, \right)_{\theta,q} = R_\Omega  (\widetilde{L}_\alpha^{p,q} (\overline{\Omega} ) ) .
\]
This completes the proof of the lemma.
\end{proof}

\subsection{Embedding  results}

We next prove several embedding results for Sobolev-Lorentz spaces $L_{\alpha}^{p,q} (\Omega )$.

\begin{lem}\label{properties of SL on domains} Let $-\infty < \alpha, \alpha_0 <\infty$, $1< p, p_0 <\infty$, and $ 1\le q ,q_0 \le \infty$.
\begin{enumerate}[{\upshape (i)}]
\item  If $\alpha  \ge \alpha_0$  and $q \le q_0 $, then   $L_\alpha^{p,q}(\Omega ) \hookrightarrow L_{\alpha_0}^{p,q_0 }(\Omega ) $. In particular, if $\alpha > 0$, then $L_\alpha^{p,q}(\Omega ) \hookrightarrow L^{p,q  }(\Omega )\hookrightarrow L_{-\alpha}^{p,q}(\Omega )$.
\item  If  $p >  p_0$,  then    $L_\alpha^{p,q}(\Omega ) \hookrightarrow L_{\alpha}^{p_0,q_0 }(\Omega ) $.
\item    If $\alpha-\alpha_0 = n/p-n/p_0 >0$, then $L_\alpha^{p,q}(\Omega )\hookrightarrow L_{\alpha_0}^{p_0 , q}(\Omega )$. In particular,  if $0<  \alpha p < n $, then $L_\alpha^{p,q}(\Omega)\hookrightarrow L^{np/(n-\alpha p),q}(\Omega)$.
\item     If $\alpha p=n$, then $L_\alpha^{p,1}(\Omega)\hookrightarrow C  (\overline{\Omega} )$ and $\widetilde{L}_\alpha^{p,1}(\Omega)\hookrightarrow C_0  (\Omega )$, where $C_0 (\Omega )$ is the closure of $C_c^\infty (\Omega )$ in $L^\infty (\Omega)$
\end{enumerate}
\end{lem}

\begin{proof} To prove Part (i), suppose that $\alpha  \ge \alpha_0$  and $q \le q_0 $. If $\alpha_0 \ge 0$, then the embeddings   $ L_\alpha^{p,q}(\Omega )  \hookrightarrow    L_{\alpha_0}^{p,q_0 }(\Omega )\hookrightarrow    L^{p,q_0 }(\Omega )$  immediately follow    from Lemma \ref{interp for SL on domains} (i) and Lemma \ref{properties of SL on Rn} (i).
Suppose that $\alpha_0 < 0$. Then since $-\alpha_0 >  0$ and $q_0 '\le q'$, it follows that  $ L_{-\alpha_0}^{p',q_0'}(\R^n )  \hookrightarrow    L^{p',q '}(\R^n ) $. Using this embedding, we easily deduce that
$$
\widetilde{L}_{-\alpha_0}^{p',q_0 '}(\Omega )     \hookrightarrow  \widetilde{L}^{p',q'}(\Omega ).
$$
Hence by the duality result  (\ref{widehat-L}),
\[
L^{p,q}(\Omega )  = \left[ \widetilde{L}^{p',q'}(\Omega ) \right]^* \hookrightarrow \left[ \widetilde{L}_{-\alpha_0}^{p',q_0 '}(\Omega ) \right]^* = L_{\alpha_0}^{p,q_0 }(\Omega )   .
\]
Therefore, if $\alpha \ge 0 > \alpha_0$, then
$L_{\alpha}^{p,q  }(\Omega ) \hookrightarrow  L^{p,q}(\Omega )   \hookrightarrow L_{\alpha_0}^{p,q_0 }(\Omega )$.
Suppose finally that $\alpha < 0$.  Then since $-\alpha_0 \ge -\alpha > 0$ and $q_0 '\le q'$, we have
$$
L_{-\alpha_0}^{p',q_0'}(\R^n )  \hookrightarrow    L_{-\alpha}^{p',q '}(\R^n )
\quad\mbox{and so}\quad
\widetilde{L}_{-\alpha_0}^{p',q_0 '}(\Omega )  \hookrightarrow  \widetilde{L}_{-\alpha}^{p',q  '}(\Omega )   .
$$
Therefore,
\[
L_\alpha^{p,q}(\Omega )  = \left[ \widetilde{L}_{-\alpha}^{p',q'}(\Omega ) \right]^* \hookrightarrow \left[ \widetilde{L}_{-\alpha_0}^{p',q_0 '}(\Omega ) \right]^* = L_{\alpha_0}^{p,q_0 }(\Omega )  .
\]
This proves Part (i).

To prove Part (ii), suppose that $p>p_0$.  The case $\alpha =0$ is trivial. Suppose that $\alpha >0$. If $k$ is a positive  integer, then $L_k^{p,p}(\Omega ) =W^{k,p}(\Omega )  \hookrightarrow  W^{k,p_0}(\Omega )= L_k^{p_0 ,p_0}(\Omega )$.
Moreover, it follows from Lemma \ref{real-interp-S on domains} (iii) that $\widetilde{L}_k^{p,p}(\Omega ) \hookrightarrow \widetilde{L}_k^{p_0 ,p_0}(\Omega )$.
Hence, if $k $ is any integer such that   $k >  \alpha $, then by   (\ref{complex interp for S}) and (\ref{comp-interp-S on domains}), we   have
$$
 L_\alpha^{p,p} (\Omega )  = \left[ L^p (\Omega ), L_k^{p}(\Omega )\right]_{\alpha /k} \hookrightarrow \left[ L^{p_0} (\Omega ), L_k^{p_0}(\Omega )\right]_{\alpha /k}=L_\alpha^{p_0 , p_0} (\Omega )
$$
and similarly
$$
\widetilde{L}_\alpha^{p,p}(\Omega ) \hookrightarrow \widetilde{L}_\alpha^{p_0 ,p_0}(\Omega ).
$$
For the general case when $q \neq p$ or $q_0 \neq p_0$, we choose $1<p_1 , p_2 , p_3 <\infty$ such that  $p_1 < p_0 < p_2 < p <p_3$, $2/p_0 =1/p_1 +1/p_2$, and $2/p =1/p_2 +1/p_3$.
Then since $L_\alpha^{p_3}(\Omega )  \hookrightarrow  L_\alpha^{p_2}(\Omega )\hookrightarrow  L_\alpha^{p_1}(\Omega )$, it follows from  Lemma \ref{interp for SL on domains} (ii) that
\[
L_\alpha^{p,q}(\Omega ) = \left(\,  L_\alpha^{p_2}(\Omega ) , L_\alpha^{p_3}(\Omega ) \, \right)_{1/2,q}  \hookrightarrow  L_\alpha^{p_2}(\Omega ) \hookrightarrow   \left(\,  L_\alpha^{p_1}(\Omega ) , L_\alpha^{p_2}(\Omega ) \, \right)_{1/2,q_0} = L_\alpha^{p_0 ,q_0 }(\Omega ).
\]
Similarly, since $\widetilde{L}_\alpha^{p_3}(\Omega )  \hookrightarrow  \widetilde{L}_\alpha^{p_2}(\Omega )\hookrightarrow  \widetilde{L}_\alpha^{p_1}(\Omega )$, it follows from  Lemma \ref{real-interp-S on domains} (i) that
\[
R_\Omega  (\widetilde{L}_\alpha^{p,q} (\overline{\Omega} ) ) \hookrightarrow  R_\Omega  (\widetilde{L}_\alpha^{p_0 ,q_0} (\overline{\Omega} ) )
\]
and hence
\[
\widetilde{L}_\alpha^{p,q}(\Omega )   \hookrightarrow \widetilde{L}_\alpha^{p_0 ,q_0 }(\Omega ).
\]

Suppose next that  $\alpha < 0$. Then since $\widetilde{L}_{-\alpha}^{p_0 ',q_0 '}(\Omega )   \hookrightarrow \widetilde{L}_{-\alpha}^{p' ,q ' }(\Omega )$ and $\widetilde{L}_{-\alpha}^{p_0 ',q_0 '}(\Omega ) $ is dense in $ \widetilde{L}_{-\alpha}^{p' ,q ' }(\Omega )$, we immediately obtain
\[
L_\alpha^{p,q}(\Omega )  = \left[ \widetilde{L}_{-\alpha}^{p',q'}(\Omega ) \right]^* \hookrightarrow \left[ \widetilde{L}_{-\alpha}^{p_0 ',q_0 '}(\Omega ) \right]^* = L_{\alpha}^{p_0 ,q_0 }(\Omega )  ,
\]
which proves Part (ii).

To prove Part (iii), suppose  that $\alpha-\alpha_0 = n/p-n/p_0 >0$. If $\alpha_0 \ge 0$, then the embeddings   $ L_\alpha^{p,q}(\Omega )  \hookrightarrow    L_{\alpha_0}^{p_0,q }(\Omega )
$  immediately follow   from Lemma \ref{interp for SL on domains} (i) and Lemma \ref{properties of SL on Rn} (ii). Suppose   that $\alpha \le 0$. Then since $(-\alpha_0) -( -\alpha) = n/p_0' - n/p' > 0$, we have
 $$
L_{-\alpha_0}^{p_0',q'}(\R^n )  \hookrightarrow    L_{-\alpha}^{p',q'}(\R^n ) ,
$$
from which we easily deduce  that
\[
\widetilde{L}_{-\alpha_0}^{p_0',q'}(\Omega )  \hookrightarrow  \widetilde{L}_{-\alpha}^{p',q  '}(\Omega )
\quad\text{and}\quad
L_\alpha^{p,q}(\Omega )
 \hookrightarrow L_{\alpha_0}^{p_0,q }(\Omega )   .
\]
If $\alpha>0 > \alpha_0$, then choosing $p_1$ such that $\alpha=n/p-n/p_1$,  we have
$$
L_\alpha^{p,q}(\Omega )
 \hookrightarrow L^{p_1,q }(\Omega )
 \hookrightarrow L_{\alpha_0}^{p_0,q }(\Omega ) .
$$
This proves Part (iii).
Finally, Part (iv) is easily deduced from  Lemma \ref{interp for SL on domains} (i)  and Lemma \ref{properties of SL on Rn} (iii).
\end{proof}


We will  also need   embedding results for Sobolev-Lorentz spaces into classical Besov spaces. For $-\infty < \alpha <\infty$, $1<p<\infty$, and $1 \le q \le \infty$, let $B_\alpha^{p,q}(\R^n )$ denote the familiar Besov space on $\R^n$. In particular, for  $0< \alpha < 1$ and $1 \le q < \infty$,    $B_\alpha^{p,q}(\R^n )$ may be defined as  the Banach space of all $f\in L^p (\R^n )$ such that
\[
\|f\|_{B_\alpha^{p,q}(\R^n )}= \|f  \|_{L^{p}(\R^n  )}  +  \left[\int_{\R^n} \frac{1}{|t|^{n+\alpha q} }\left( \int_{\R^n} |f (x+t) -f (x)|^{p}\, dx \right)^{q/p} \, dt \right]^{1/q}
\]
is finite (see, e.g., \cite[p. 172]{JK}). If  $0< \alpha < 1$ and $  q = \infty$, then the corresponding norm  is
\[
\|f\|_{B_\alpha^{p,\infty}(\R^n )}= \|f  \|_{L^{p}(\R^n  )}  + \sup_{t \neq 0}\frac{1}{|t|^\alpha}\left( \int_{\R^n} |f (x+t) -f (x)|^{p}\, dx \right)^{1/p}  .
\]
Recall also from \cite[(2.13)]{JK} that
 for any real $\alpha$,
$f \in B_\alpha^{p,q}(\R^n )$ if and only if $f \in B_{\alpha-1}^{p,q}(\R^n )$ and $\nabla f \in B_{\alpha-1}^{p,q}(\R^n ;\R^n )$.
For $\alpha \ge 0$,   we define
$$
B_\alpha^{p,q}(\Omega ) = \left\{ R_\Omega f \,:\, f \in  B_\alpha^{p,q}(\R^n )\right\} ,
$$
which is a Banach space equipped with the usual quotient norm. Then $R_\Omega$ is bounded from  $B_\alpha^{p,q}(\R^n )$ into $B_\alpha^{p,q}(\Omega )$ and $E_\Omega$ is bounded from $B_\alpha^{p,q}(\Omega )$  into $B_\alpha^{p,q}(\R^n )$ (see \cite[Proposition 2.17]{JK}).

\begin{lem}\label{embedding of L into B on domain} Let $0 \le \alpha, \alpha_0 <\infty$, $1< p, p_0 <\infty$, and $ 1\le q   \le \infty$.
   If $\alpha-\alpha_0 = n/p-n/p_0 >0$, then
$$
L_\alpha^{p,q}(\R^n )\hookrightarrow B_{\alpha_0}^{p_0 , q}(\R^n )
\quad\mbox{and}\quad L_\alpha^{p,q}(\Omega )\hookrightarrow B_{\alpha_0}^{p_0 , q}(\Omega ).
$$
\end{lem}

\begin{proof} The lemma is   an immediate  consequence of a general embedding result  due to  Seeger and Trebels \cite{seeger}, who indeed   established complete  embedding results for   Tribel-Lizorkin-Lorentz spaces $F_r^\alpha [L^{p,q}(\R^n )]$ and   Besov-Lorentz spaces $B_q^\alpha [L^{p,r}(\R^n )]$.

Suppose that $\alpha-\alpha_0 =n/p-n/p_0 >0$. Then  it follows from  \cite[Theorem 1.2]{seeger} that
\begin{equation}\label{embedding of F into B}
F_r^\alpha [L^{p,q}(\R^n )] \hookrightarrow B_q^{\alpha_0} [L^{p_0, r_0}(\R^n )] \quad\mbox{for any}\,\, 1 \le r, r_0 \le \infty.
\end{equation}
Note that
$$
F_2^\alpha [L^{p,q}(\R^n )] =L_\alpha^{p,q} (\R^n )\quad\mbox{and}\quad
B_q^{\alpha_0} [L^{p_0,p_0}(\R^n )]= B_{\alpha_0}^{p_0, q}(\R^n )
$$
(see \cite[p. 1018]{seeger}). Hence by (\ref{embedding of F into B}), we  have
\[
L_\alpha^{p,q}(\R^n ) \hookrightarrow B_{\alpha_0}^{p_0 ,q}(\R^n ),
\]
which   implies, by the   boundeness  of $R_\Omega$ and $E_\Omega$, that
$L_\alpha^{p,q}(\Omega )   \hookrightarrow B_{\alpha_0}^{p_0 ,q}(\Omega )$.
\end{proof}

\subsection{Trace results}\label{subsection: trace}

The trace operator ${\rm Tr}$ is defined initially on $C^\infty(\overline{\Omega})$ by
\[
{\rm Tr}\,  u = u |_{\partial\Omega}\quad \text{for all } u\in C^\infty(\overline{\Omega}).
\]
Let $ 1<p<\infty$,  $ 1/p <\alpha< \infty$, and $1 \le q \le \infty$.
Choose $1<p_0 , p_1 < \infty$ and $0<\theta< 1$ such that $1/ p_1 < 1/p < 1/p_0 < \alpha$ and $1/p = (1-\theta)/p_0 +\theta /p_1$. Then it follows from the classical trace theorem (see, e.g., \cite[Theorems 1 and 2]{J}, \cite[Chapter VII]{JW}, and \cite[Theorem 3.1]{JK}) that for each $i=0,1$, ${\rm Tr}$ extends uniquely to a bounded linear operator $T_i$ from  $L_\alpha^{p_i} (\Omega )$ into $L^{p_i} (\partial\Omega )$.
Note that $L_\alpha^{p_0} (\Omega )\cap L_\alpha^{p_1} (\Omega )=L_\alpha^{p_1} (\Omega )$ and $T_0 =T_1$ on $L_\alpha^{p_1} (\Omega )$.
Hence by Lemma \ref{interp for SL on domains} (ii), ${\rm Tr}$ extends uniquely to a bounded linear operator   from  $L_\alpha^{p,q} (\Omega )$ into $L^{p,q} (\partial\Omega )$, still denoted by ${\rm Tr}$.  Let  $L_{\alpha,0}^{p,q} (\Omega )$  be  the kernel of the trace operator ${\rm Tr}: L_\alpha^{p,q} (\Omega ) \rightarrow L^{p,q} (\partial\Omega )$:
\[
L_{\alpha,0}^{p,q} (\Omega ) =\left\{ u \in L_\alpha^{p,q} (\Omega )\,:\, {\rm Tr} \, u =0 \,\,\mbox{on}\,\,\partial\Omega \right\}.
\]
For $1<p=q<\infty$, we write $L_{\alpha,0}^{p} (\Omega ) = L_{\alpha,0}^{p,p} (\Omega )$. Then it follows from Lemma \ref{real-interp-S on domains} (ii) and     \cite[Proposition 3.3]{JK} that
\begin{equation}\label{hat-Lpalpha}
\widetilde{L}_{\alpha}^{p}(\Omega) =
 L_{\alpha,0}^{p} (\Omega )  \quad \mbox{for}\,\,1/p<\alpha < 1+1/p .
\end{equation}

For $s > 0$, $1<p<\infty$, and $1\le q \le \infty$,
let  $\mathcal B^{p,q}_{s}(\partial \Omega )$ denote the range  of ${\rm Tr}:  L_{s+1/p}^{p,q} (\Omega )  \to L^{p,q}(\partial \Omega)$, which is a Banach space  equipped with the usual   quotient norm. Then  it follows from Lemma \ref{interp for SL on domains}  (iv) that if $C^\infty (\partial\Omega ) = \left\{  u|_{\partial\Omega}\,:\, u \in C_c^\infty (\R^n ) \right\}$, then $C^\infty (\partial\Omega )$ is dense in  $\mathcal B^{p,q}_{s}(\partial \Omega )$ for   $1 \le q<\infty$. If $s>0$, $1<p<\infty$, and $1 < q \le \infty$, we denote by $\mathcal{B}_{-s}^{p,q}(\partial\Omega )$ the dual space of $\mathcal{B}_{s}^{p' ,q'}(\partial\Omega )$.
For    $1<p=q<\infty$, we write  $\mathcal B^{p}_{s}(\partial \Omega )= \mathcal B^{p,p}_{s}(\partial \Omega )$. It is well known (see, e.g.,  \cite[Theorem 3.1]{JK}) that if $0< s < 1$ and $1< p<\infty$, then  $\mathcal B^{p}_{s}(\partial \Omega )$ coincides with the standard boundary Besov space $B_s^p (\partial \Omega )$.
In general,  $\mathcal B^{p,q}_{s}(\partial \Omega )$ is a real interpolation space of two boundary Besov spaces.

\begin{lem}\label{trace results-1}  Let $1<p<\infty$, $1/p < \alpha <1 +1/p$, and $1 \le q \le \infty$.
\begin{enumerate}[{\upshape(i)}]
\item  There is a bounded linear operator ${\rm Ex} : \mathcal B^{p,q}_{\alpha-1/p}(\partial \Omega ) \rightarrow L_\alpha^{p,q} (\Omega )$ such that ${\rm Tr} \circ {\rm Ex}  $ is  the identity on $\mathcal B^{p,q}_{\alpha-1/p}(\partial \Omega )$.
\item If  $1< p_0 < p < p_1 <\infty$, $1/p_0< \alpha <1+1/p_1$, $0< \theta <1$, and $1/p=(1-\theta)/p_0+\theta/p_1$,
then
\[
\left(  B^{p_0}_{\alpha-1/p_0}(\partial \Omega ),  B^{p_1}_{\alpha-1/p_1}(\partial \Omega ) \right)_{\theta,q}= \mathcal B^{p,q}_{\alpha-1/p}(\partial \Omega ).
\]
\end{enumerate}
\end{lem}

\begin{proof}
Suppose   that  $1< p_0 < p < p_1 <\infty$, $1/p_0< \alpha <1+1/p_1$, $0< \theta <1$, and $1/p=(1-\theta)/p_0+\theta/p_1$.

Since    ${\rm Tr}$ is bounded from $L^{p_i}_{\alpha}(\Omega)$ into $ B^{p_i}_{\alpha-1/p_i}(\partial \Omega)$ for $i=0,1$,   it follows from Lemma \ref{interp for SL on domains} (ii) that  ${\rm Tr}$ is bounded from $L^{p,q}_{\alpha}(\Omega)$ into $\left(  B^{p_0}_{\alpha-1/p_0}(\partial \Omega ),  B^{p_1}_{\alpha-1/p_1}(\partial \Omega ) \right)_{\theta,q}$. By the definition of $\mathcal{B}^{p,q}_{\alpha-1/p}(\partial \Omega )$, we have
\[
 \mathcal B^{p,q}_{\alpha-1/p}(\partial \Omega ) \hookrightarrow\left(  B^{p_0}_{\alpha-1/p_0}(\partial \Omega ),  B^{p_1}_{\alpha-1/p_1}(\partial \Omega ) \right)_{\theta,q}.
\]
To prove the reverse embedding, we recall from \cite[Theorem 3.1]{JK}  that
there is a   linear extension operator ${\rm Ex}$ such that ${\rm Ex}$ maps $B^r_s(\partial \Omega)$ into $L^r_{s+1/r}(\Omega)$ boundedly   and ${\rm  Tr}\circ {\rm Ex}$ is the identity operator on $B^r_s(\partial \Omega)$  for all $0<s<1$  and $1<r<\infty$.   By  Lemma \ref{interp for SL on domains} (ii) again, ${\rm Ex}$ is bounded from $\left(  B^{p_0}_{\alpha-1/p_0}(\partial \Omega ),  B^{p_1}_{\alpha-1/p_1}(\partial \Omega ) \right)_{\theta,q}$ into $L^{p,q}_{\alpha}(\Omega)$. Since ${\rm  Tr}\circ {\rm Ex}$ is the identity and ${\rm  Tr} $ is bounded from $L^{p,q}_{\alpha}(\Omega)$ onto $\mathcal{B}^{p,q}_{\alpha-1/p}(\partial \Omega )$, we obtain
\[
\left(  B^{p_0}_{\alpha-1/p_0}(\partial \Omega ),  B^{p_1}_{\alpha-1/p_1}(\partial \Omega ) \right)_{\theta,q} \hookrightarrow \mathcal B^{p,q}_{\alpha-1/p}(\partial \Omega ).
\]
This completes the proof of the lemma.
\end{proof}

\begin{rmk}
The interpolation result  in Lemma  \ref{trace results-1} should be compared with    the following classical  one (see  \cite[(2.15)]{JK}):
$$
\big(\,B^{p,q_1}_{\alpha_1}(\R^n),B^{p,q_2}_{\alpha_2}(\R^n)\,\big)_{\theta,q}= B^{p,q}_{\alpha}(\R^n)
$$
whenever $-\infty < \alpha_1\neq \alpha_2 < \infty$, $0<\theta<1$, $\alpha=(1-\theta)\alpha_0+\theta\alpha_1$, and $1\le p,q_1,q_2,q\le \infty$.
\end{rmk}

\section{Poisson equation}

In this section, we consider the   Dirichlet problem for the Poisson equation:
\begin{equation}\label{eq:D-1}
\left\{\begin{alignedat}{2}
- \Delta u  & =f & \quad & \text{in }\Omega,\\
  u & =0 & \quad & \text{on }\partial\Omega ,
\end{alignedat}\right.
\end{equation}
where $\Omega$ is a bounded Lipschitz domain in $\R^n , n \ge 2$.

The  Laplacian $\Delta$ is defined initially on $C^\infty(\overline{\Omega})$.
Let $1<p<\infty$. Then by H\"{o}lder's inequality,
\[
\int_\Omega (-\Delta u) \psi \,d x = \int_\Omega u (-\Delta \psi )  \,d x \le C \|u\|_{L^{p} (\Omega )}\|\psi \|_{L_2^{p'}(\Omega )}.
\]
for all $u\in C^\infty(\overline{\Omega})$ and $\psi \in C_c^\infty (\Omega )$. Hence  $-\Delta$ can be extended  uniquely to a bounded linear operator from $L^{p}(\Omega )$ into $L_{-2}^{p}(\Omega )$, by defining
\[
\langle -\Delta u , g \rangle = \int_\Omega u (-\Delta g )  \,d x \quad\mbox{for all}\,\, u \in L^{p}(\Omega )\,\,\mbox{and}\,\, g \in \widetilde{L}_{2}^{p'}(\Omega ).
\]
Obviously, $-\Delta$ is bounded from $L_2^p (\Omega )$ into $L^p (\Omega )$.
On the other hand, it  follows from  (\ref{complex interp for S}) and (\ref{comp-interp-S on domains})  that
$$
\left[ \, L^{p} (\Omega ) , L_2^{p} (\Omega)  \,\right]_{\alpha /2} = L_{\alpha}^{p}(\Omega ) \quad\mbox{and}\quad \left[ \, L_{  -2}^{p} (\Omega ) , L^{p} (\Omega)  \,\right]_{\alpha /2} =L_{\alpha-2}^p (\Omega )
$$
for $0<\alpha<2$. Hence,  by complex interpolation,  $-\Delta$ is bounded from $L_{\alpha}^{p}(\Omega )$ into $L_{\alpha-2}^p (\Omega )$ for   $0 \le \alpha \le   2$. Moreover, by real interpolation, it follows from Lemmas \ref{interp for SL on domains} and \ref{real-interp-S on domains} that $-\Delta$ is bounded from $L_{\alpha}^{p,q}(\Omega )$ into $L_{\alpha-2}^{p,q} (\Omega )$ for   $0 \le \alpha \le   2$, $1<p<\infty$, and $1 \le q\le \infty$.

\subsection{Solvability  for data in $L_{\alpha-2}^{p,q}(\Omega )$ with $1/p < \alpha \le 2$}

 We start with   unique solvability results for the Poisson equation  with  data in $L_{\alpha-2}^{p}(\Omega )$ with $1/p < \alpha \le 2$.

\begin{lem}\label{Poisson-S spaces}
\begin{enumerate}[{\upshape(i)}]
\item  Let $\Omega$ be a bounded $C^1$-domain in $\R^n   $, $n\ge 2$. Assume that $1<p<\infty$ and $1/p <\alpha<1+ 1/p$.
Then for every $f\in L_{\alpha-2}^{p}(\Omega )$, there exists a unique solution $u\in L_{\alpha, 0}^{p}(\Omega )$ of (\ref{eq:D-1}). Moreover,
\[
\|u\|_{L_{\alpha}^{p}(\Omega )} \le C \|f\|_{L_{\alpha-2}^{p}(\Omega )}
\]
for some $C=C(n,\alpha ,p, \Omega )$.
\item Let $\Omega$ be a bounded $C^{1,1}$-domain in $\R^n   $, $n\ge 2$.
Assume that $1<p<\infty$ and $  1   \le \alpha \le 2$. Then for every $f\in L_{\alpha-2}^{p}(\Omega )$, there exists a unique solution $u\in L_{1 , 0}^{p}(\Omega ) \cap L_{\alpha}^{p}(\Omega )$ of (\ref{eq:D-1}). Moreover,
\[
\|u\|_{L_{\alpha}^{p}(\Omega )} \le C \|f\|_{L_{\alpha-2}^{p}(\Omega )}
\]
for some $C=C(n, \alpha ,p, \Omega )$.
\end{enumerate}
\end{lem}

\begin{proof}
 Part (i) is a very special case of \cite[Theorems 1.1 and 1.3]{JK}.
To prove Part (ii), suppose that $\Omega$ is of class $C^{1,1}$, $1<p<\infty$,   and  $  1  \le  \alpha \le 2$. The case $\alpha =1$ is covered by Part (i) and the case $\alpha = 2$ follows    from  the classical Calderon-Zygmund theory.
Suppose that $1< \alpha < 2$.
Then it follows from Part (i)  with $\alpha=1$
that for every $f \in L_{-1}^{p} (\Omega )$, there exists a unique solution $u \in L_{1 ,0}^{p} (\Omega )$ of  (\ref{eq:D-1}). Let   $T : L_{-1}^{p} (\Omega ) \to L_{1 ,0}^{p} (\Omega )$ be the associated  solution operator for (\ref{eq:D-1}), which is linear and bounded.
Since $L^p_{\alpha-2}(\Omega)\hookrightarrow L_{  -1}^{p} (\Omega )$, $T$ is bounded from  $L^p_{\alpha-2}(\Omega)$ into $L_{1 ,0}^{p} (\Omega )$.
By the Calderon-Zygmund theory, $T$ is also bounded from $L^{p} (\Omega)$ into $L_2^{p} (\Omega )$. Moreover, it   follows from (\ref{comp-interp-S on domains}) and   (\ref{complex interp for S}) that
$$
L_{\alpha -2}^{p} (\Omega )= \left[ \, L_{  -1}^{p} (\Omega ) , L^{p} (\Omega)  \,\right]_{\alpha-1}  \quad\mbox{and}\quad \left[ \, L_{1}^{p} (\Omega ) , L_2^{p} (\Omega) \,\right]_{\alpha-1} = L_{\alpha}^{p} (\Omega ).
$$
Therefore, by complex interpolation, we conclude that    $T$ is bounded from $L_{\alpha -2}^{p} (\Omega )$ into $ L_\alpha^{p} (\Omega )$.
This completes the proof of the lemma.
\end{proof}

 Using the above solvability result, we obtain the following interpolation results.

\begin{lem}\label{Interpolatoin of data-solution}
\begin{enumerate}[{\upshape(i)}]
\item  Let $\Omega$ be a bounded $C^1$-domain in $\R^n $, $n\ge 2$. Assume that $1<p_0 < p_1 <\infty$, $1/p_0 <\alpha<1+ 1/p_1$,  $0< \theta < 1$, $1/p=(1-\theta)/p_0 +\theta/p_1$, and $1 \le q \le \infty$.
Then
\[
\left(\,  L_{\alpha,0}^{p_0}(\Omega ) , L_{\alpha,0}^{p_1}(\Omega ) \, \right)_{\theta,q} = L_{\alpha,0}^{p,q}(\Omega ) .
\]
\item Let $\Omega$ be a bounded $C^{1,1}$-domain in $\R^n$, $n\ge 2$.  Assume that $1<p_0 \neq  p_1 <\infty$, $0< \theta < 1$, $1/p=(1-\theta)/p_0 +\theta/p_1$, $1 \le \alpha \le 2$, and $1 \le q \le \infty$.
Then
\[
\left(\, L_{1 ,0}^{p_0}(\Omega ) \cap  L_{\alpha}^{p_0}(\Omega ) , L_{1,0}^{p_1}(\Omega ) \cap L_{\alpha}^{p_1}(\Omega ) \, \right)_{\theta,q} = L_{1,0}^{p,q}(\Omega ) \cap L_{\alpha}^{p,q}(\Omega ).
\]
\end{enumerate}
\end{lem}

\begin{proof}
We first  prove Part (i).    Suppose that $\Omega$ is of class $C^1$, $1<p_0 < p_1 <\infty$, $1/p_0 <\alpha<1+ 1/p_1$, $0< \theta < 1$, $1/p=(1-\theta)/p_0 +\theta/p_1$, and $1 \le q \le \infty$.
Then it follows from Lemma \ref{Poisson-S spaces} (i)  that for every $f \in L_{\alpha-2}^{p_0} (\Omega )$, there exists a unique solution $u \in L_{\alpha,0}^{p_0} (\Omega )$ of  (\ref{eq:D-1}). Let   $T : L_{\alpha-2}^{p_0} (\Omega ) \to L_{\alpha,0}^{p_0} (\Omega )$ be the associated  solution operator for (\ref{eq:D-1}).  By Lemma \ref{Poisson-S spaces} (i) again, $T$ is   bounded from  $L_{\alpha-2}^{p_1} (\Omega )$ into $ L_{\alpha,0}^{p_1} (\Omega )$.
Hence by  Lemma \ref{real-interp-S on domains} (i), $T$ is bounded from $   L_{\alpha-2}^{p,q}(\Omega ) $ into $\left(\,  L_{\alpha,0}^{p_0}(\Omega ) , L_{\alpha,0}^{p_1}(\Omega ) \, \right)_{\theta,q}$.
 Suppose that  $u \in L_{\alpha, 0}^{p,q} (\Omega )$ and $f= -\Delta u$. Then $f \in  L_{\alpha-2}^{p,q}(\Omega )$ and $Tf \in \left(\,  L_{\alpha,0}^{p_0}(\Omega ) , L_{\alpha,0}^{p_1}(\Omega ) \, \right)_{\theta,q}$. Since $u, Tf \in L_{\alpha,0}^{p_0}(\Omega )$ and $-\Delta u =-\Delta (Tf)$ in $\Omega$, it follows from Lemma \ref{Poisson-S spaces} (i) that
$u =Tf$ and
$$
\|u\|_{\left(\,  L_{\alpha,0}^{p_0}(\Omega ) , L_{\alpha,0}^{p_1}(\Omega ) \, \right)_{\theta,q}} \lesssim \|f\|_{L_{\alpha-2}^{p,q}(\Omega )} \lesssim  \|u\|_{L_{\alpha, 0}^{p,q} (\Omega )}.
$$
This implies that $L_{\alpha, 0}^{p,q} (\Omega ) \hookrightarrow \left(\,  L_{\alpha,0}^{p_0}(\Omega ) , L_{\alpha,0}^{p_1}(\Omega ) \, \right)_{\theta,q}$. To prove the reverse embedding, we simply observe  that
$$
\left(\,  L_{\alpha,0}^{p_0}(\Omega ) , L_{\alpha,0}^{p_1}(\Omega ) \, \right)_{\theta,q} \hookrightarrow \left(\,  L_{\alpha}^{p_0}(\Omega ) , L_{\alpha}^{p_1}(\Omega ) \, \right)_{\theta,q} = L_{\alpha}^{p,q} (\Omega )
,
$$
$$
\left(\,  L_{\alpha,0}^{p_0}(\Omega ) , L_{\alpha,0}^{p_1}(\Omega ) \, \right)_{\theta,q} \hookrightarrow L_{\alpha,0}^{p_0}(\Omega ),\quad\mbox{and}\quad
L_{\alpha}^{p,q} (\Omega )\cap  L_{\alpha,0}^{p_0}(\Omega ) = L_{\alpha,0}^{p,q}(\Omega ).
$$
This proves Part (i).

We next prove Part (ii).   Suppose  that $\Omega$ is of class $C^{1,1}$,  $1<p_0 <  p_1 <\infty$, $0< \theta < 1$, $1/p=(1-\theta)/p_0 +\theta/p_1$, $1 \le \alpha \le 2$, and $1 \le q \le \infty$. Then since $1/p_0 < 1 < 1+1/p_1$, it follows from Part (i) and Lemma \ref{interp for SL on domains} (ii) that
\[
\left(\, L_{1 ,0}^{p_0}(\Omega ) \cap  L_{\alpha}^{p_0}(\Omega ) , L_{1,0}^{p_1}(\Omega ) \cap L_{\alpha}^{p_1}(\Omega ) \, \right)_{\theta,q} \hookrightarrow   L_{1 ,0}^{p, q}(\Omega )   \cap     L_{\alpha}^{p, q}(\Omega ) .
\]
To prove the reverse embedding, let    $T : L_{\alpha-2}^{p_0} (\Omega ) \to L_{1 ,0}^{p_0}(\Omega ) \cap L_{\alpha}^{p_0} (\Omega )$ be the solution operator for (\ref{eq:D-1}), which is bounded by Lemma \ref{Poisson-S spaces}. Then by Lemma \ref{real-interp-S on domains} (i), $T$ is also bounded from $  L_{\alpha-2}^{p,q}(\Omega ) $ into $\left(\, L_{1 ,0}^{p_0}(\Omega ) \cap  L_{\alpha}^{p_0}(\Omega ) , L_{1,0}^{p_1}(\Omega ) \cap L_{\alpha}^{p_1}(\Omega ) \, \right)_{\theta,q}$. Suppose that $u \in L_{1 ,0}^{p, q}(\Omega )   \cap     L_{\alpha}^{p, q}(\Omega )$ and $f= -\Delta u$. Then  $f \in    L_{\alpha-2}^{p,q}(\Omega ) $,  $u, Tf \in L_{1 ,0}^{p_0}(\Omega )$, and $-\Delta u =-\Delta Tf$ in $\Omega$. Hence by Lemma \ref{Poisson-S spaces} (i), we have
\[
u =Tf \in \left(\, L_{1 ,0}^{p_0}(\Omega ) \cap  L_{\alpha}^{p_0}(\Omega ) , L_{1,0}^{p_1}(\Omega ) \cap L_{\alpha}^{p_1}(\Omega ) \, \right)_{\theta,q}
\]
and
\[
\|u\|_{\left(\, L_{1 ,0}^{p_0}(\Omega ) \cap  L_{\alpha}^{p_0}(\Omega ) , L_{1,0}^{p_1}(\Omega ) \cap L_{\alpha}^{p_1}(\Omega ) \, \right)_{\theta,q}} \lesssim  \|f\|_{     L_{\alpha-2}^{p, q}(\Omega )} \lesssim \|u\|_{     L_{\alpha}^{p, q}(\Omega )}.
\]
This completes the proof of the lemma.
\end{proof}


Immediately from Lemma  \ref{real-interp-S on domains}, (\ref{hat-Lpalpha}), and Lemma  \ref{Interpolatoin of data-solution}, we    obtain the following  result.

\begin{lem}\label{R-tilde equal zero trac}  Let $\Omega$ be  a bounded $C^1$-domain in $\R^n , n \ge 2  $. If   $1<p<\infty$, $1/p <\alpha<1+ 1/p$, and $1 \le q \le \infty$,  then
$R_\Omega  ( \widetilde{L}_\alpha^{p,q}(\overline{\Omega} )) = L_{\alpha, 0}^{p,q}(\Omega )$. In addition, if $ q<\infty$, then
$\widetilde{L}_\alpha^{p,q}(\Omega ) = L_{\alpha, 0}^{p,q}(\Omega )$.
\end{lem}

We are ready to prove  the following  unique solvability results for the Poisson equation  with  data in $L_{\alpha-2}^{p,q}(\Omega )$, where $1/p< \alpha \le 2$.

\begin{thm}\label{Poisson-FSL spaces}
\begin{enumerate}[{\upshape(i)}]
\item  Let $\Omega$ be a bounded $C^1$-domain in $\R^n , n \ge 2  $. Assume that
$1<p<\infty$, $1/p <\alpha<1+ 1/p$, and $1 \le q \le \infty$.
Then for every $f\in L_{\alpha-2}^{p,q}(\Omega )$, there exists a unique solution $u\in L_{\alpha, 0}^{p,q}(\Omega )$ of (\ref{eq:D-1}).  Moreover,
\[
\|u\|_{L_{\alpha}^{p,q}(\Omega )} \le C \|f\|_{L_{\alpha-2}^{p,q}(\Omega )}
\]
for some $C=C(n,\alpha ,p, q, \Omega )$.
\item Let $\Omega$ be a bounded $C^{1,1}$-domain in $\R^n  , n \ge 2$. Assume that $1<p<\infty$, $ 1   \le \alpha \le 2$, and $1 \le q \le \infty$.
Then for every $f\in L_{\alpha-2}^{p,q}(\Omega )$, there exists a unique solution $u\in L_{1 , 0}^{p,q}(\Omega ) \cap L_{\alpha}^{p,q}(\Omega )$ of (\ref{eq:D-1}).  Moreover,
\[
\|u\|_{L_{\alpha}^{p,q}(\Omega )} \le C \|f\|_{L_{\alpha-2}^{p,q}(\Omega )}
\]
for some $C=C(n,\alpha ,p, q, \Omega )$.
\end{enumerate}
\end{thm}

\begin{proof}
Suppose that $\Omega$ is of class  $C^1$, $1<p<\infty$, $1/p <\alpha<1+ 1/p$, and $1 \le q \le \infty$. Choose $1<p_0 < p_1 <\infty$ and  $0< \theta < 1$ such that  $1/p=(1-\theta)/p_0 +\theta/p_1$ and  $1/p_0 <\alpha<1+ 1/p_1$.
Let   $T : L_{\alpha-2}^{p_0} (\Omega ) \to L_{\alpha,0}^{p_0} (\Omega )$ be the    solution operator for (\ref{eq:D-1}). Then by the proof of Lemma \ref{Interpolatoin of data-solution} (i),  $T$   is bounded from $L_{\alpha-2}^{p,q}(\Omega )$ into $L_{\alpha, 0}^{p,q}(\Omega )$, which  proves the existence   assertion and estimate of Part (i). The uniqueness follows from Lemma \ref{Poisson-S spaces} (i). This proves Part (i).  Part (ii) is similarly proved by using the proof of Lemma \ref{Interpolatoin of data-solution} (ii).
\end{proof}

\subsection{Solvability   for measure data}

Let $\Omega$ be a bounded domain in $\R^n$.
We denote by $\mathcal{M}  (\Omega )$ the space  of all signed  Radon measures $\mu$ on $\Omega$ with finite total variation $  |\mu |(\Omega )$. Recall (see, e.g.,  \cite[Chapter 7]{Folland}) that
\begin{enumerate}[{\upshape(i)}]
\item  $\mathcal{M}  (\Omega )$ is a Banach space equipped with the norm
$\|\mu \|_{\mathcal{M}  (\Omega )}=|\mu |(\Omega )$;
\item $L^1 (\Omega ) \hookrightarrow \mathcal{M}  (\Omega )$: more precisely, the mapping $f \mapsto \mu_f$, defined by
\[
\mu_f (A) =\int_A f \, dx \quad\mbox{for all Borel sets}\,\,A \subset \Omega ,
\]
is an isometry from $L^1 (\Omega )$ into  $\mathcal{M}  (\Omega )$;
\item $  \mathcal{M}  (\Omega )$ is the dual space of
$C_0 (\Omega )$: more precisely, the mapping $\mu \mapsto f_\mu$, defined by
\[
\langle f_\mu ,\phi \rangle =\int_\Omega \phi \, d\mu \quad\mbox{for all}\,\,\phi \in C_0 (\Omega),
\]
is an isometric isomorphism from  $\mathcal{M}  (\Omega )$ onto $\left[ C_0 (\Omega ) \right]^*$.
\end{enumerate}

\medskip
Recall from  (\ref{def-pa}) that
\[
p(\alpha) =\frac{n}{n-1+\alpha} \quad\mbox{for}\,\,  0 \le   \alpha < 1.
\]
Note that $1< p(\alpha) \le  n/(n-1)$ and
$\beta -\alpha  = n/p(\beta) -n/p(\alpha)  >0$ if $0 \le  \alpha <  \beta  < 1$.
Hence it follows from Lemmas \ref{properties of SL on domains} and \ref{embedding of L into B on domain} that if $0 \le  \alpha <  \beta  < 1$ and $\gamma \ge 0$, then
\begin{equation}\label{embedding of L into L and B}
L_{\beta+\gamma}^{p(\beta),\infty}(\Omega )\hookrightarrow L_{\alpha+\gamma}^{p(\alpha),\infty}(\Omega ) \cap B_{\alpha+\gamma}^{p(\alpha),\infty}(\Omega ).
\end{equation}

\begin{lem}\label{embedding of M} Let $\Omega$ be a  bounded Lipschitz domain in $\R^n$. Then for $0 \le  \alpha < 1$, we have
\[
\mathcal{M}  (\Omega ) \hookrightarrow L_{\alpha-1}^{p(\alpha),\infty}(\Omega ) .
\]
\end{lem}

\begin{proof} Let $ \mu\in \mathcal{M}  (\Omega )$ be given.
Then since
$(1-\alpha ) p(\alpha) ' =n$, it follows from Lemma \ref{properties of SL on domains} (iv) that
$
\widetilde{L}_{1-\alpha}^{p(\alpha)',1}(\Omega ) \hookrightarrow C_0 (\Omega )$.
Hence for all $\phi \in \widetilde{L}_{1-\alpha}^{p(\alpha)',1}(\Omega ) $, we have
\[
\langle f_\mu , \phi \rangle = \int_\Omega \phi \, d \mu \le \|\phi \|_{L^\infty (\Omega )}  \|\mu \|_{\mathcal{M}  (\Omega )} \lesssim  \|\phi \|_{\widetilde{L}_{1-\alpha}^{p(\alpha)',1}(\Omega )}  \|\mu \|_{\mathcal{M}  (\Omega )},
\]
which implies that
$$
f_\mu \in L_{\alpha-1}^{p(\alpha),\infty}(\Omega ) = \left[\widetilde{L}_{1-\alpha}^{p(\alpha)',1}(\Omega )\right]^*  \quad\mbox{and}\quad
 \|f_\mu \|_{L_{\alpha-1}^{p(\alpha),\infty}(\Omega )}\lesssim  \|\mu \|_{\mathcal{M}  (\Omega )} .
$$
This completes the proof.
\end{proof}

\begin{thm}\label{Poisson-measure data}
Let $\Omega$ be a bounded $C^1$-domain in $\R^n , n \ge 2 $.
Then for every  $ f\in L_{-1}^{n/(n-1),\infty} (\Omega )$, there exists a unique solution  $u\in L_{1, 0}^{n/(n-1),\infty}(\Omega )$  of (\ref{eq:D-1}), which satisfies
\begin{equation}\label{est-u-L1p}
\|u\|_{L_{1}^{n/(n-1),\infty}(\Omega )}  \le C \|f \|_{L_{-1}^{n/(n-1),\infty} (\Omega )}
\end{equation}
for some $C=C(n,\Omega )$. In addition, if $ f\in \mathcal{M}  (\Omega )$, then for every $0 \le   \alpha <1$,
\[
\widetilde{u}  \in L_{\alpha+1}^{p(\alpha) ,\infty}(\R^n  )\cap B_{\alpha+1}^{p(\alpha) ,\infty}(\R^n  ) \quad\mbox{and}\quad  \|\widetilde{u}  \|_{L_{\alpha+1}^{p(\alpha),\infty}(\R^n )\cap B_{\alpha+1}^{p(\alpha),\infty}(\R^n )} \le C \|f \|_{\mathcal{M}  (\Omega )},
\]
where  $C=C(n,\alpha ,  \Omega )$. Here $\widetilde{u}$ denotes the zero extension of $u$ to $\R^n$.
\end{thm}

\begin{proof}  Let $ f\in L_{-1}^{n/(n-1),\infty} (\Omega )$ be given. Then it immediately  follows from   Theorem \ref{Poisson-FSL spaces} (i) that  there exists   a unique solution  $u\in L_{1, 0}^{n/(n-1),\infty}(\Omega )$  of (\ref{eq:D-1}), which satisfies the estimate  (\ref{est-u-L1p}).

Suppose in addition that  $ f\in \mathcal{M}  (\Omega )$.
Note that if $0 \le    \alpha <1$, then  $1< p(\alpha)<\infty$ and $1/p(\alpha) < \alpha +1 < 1+1/p(\alpha)$. Hence it follows from Lemma~\ref{embedding of M} and Theorem \ref{Poisson-FSL spaces} (i) that for every  $0 \le  \alpha <1$, there exists a unique solution $u_\alpha \in L_{\alpha+1,0}^{p(\alpha) ,\infty}(\Omega )$ of (\ref{eq:D-1}), which satisfies
\begin{equation}\label{est-u-alpha}
\|u_\alpha \|_{L_{\alpha+1}^{p(\alpha),\infty}(\Omega )} \lesssim \|f \|_{L_{\alpha-1}^{p(\alpha),\infty}(\Omega )} \lesssim  \|f \|_{\mathcal{M}  (\Omega )} .
\end{equation}
Let   $0 \le   \alpha <1$ be fixed. Choose any $\beta$ such that
 $  \alpha   <  \beta  < 1$. Then by (\ref{embedding of L into L and B}), we have
\[
u_\beta \in L_{\beta+1,0}^{p(\beta),\infty}(\Omega )\hookrightarrow L_{1,0}^{n/(n-1),\infty}(\Omega )   .
\]
Since  $u   , u_\beta     \in L_{1,0}^{n/(n-1),\infty}(\Omega )$,
it follows from the uniqueness assertion of Theorem \ref{Poisson-FSL spaces} (i) that $u  =  u_\beta  $. Moreover, since $1/p(\beta)< \beta+1 < 1+1/
p(\beta)$,   it follows from Lemma \ref{R-tilde equal zero trac} that
$$
u  \in L_{\beta+1,0}^{p(\beta),\infty}(\Omega ) =  R_\Omega  ( \widetilde{L}_{\beta+1}^{p(\beta),\infty}(\overline{\Omega} )).
$$
Hence
$$
\widetilde{u}  \in L_{\beta+1}^{p(\beta) ,\infty}(\R^n  )
\quad\mbox{and}\quad \|\widetilde{u}  \|_{L_{\beta+1}^{p(\beta) ,\infty}(\R^n  )} \approx \|u   \|_{L_{\beta+1}^{p(\beta),\infty}(\Omega ) }.
$$
Therefore, by Lemma \ref{properties of SL on Rn} (ii),  Lemma \ref{embedding of L into B on domain}, and (\ref{est-u-alpha}), we have
\[
\widetilde{u}  \in L_{\alpha+1}^{p(\alpha) ,\infty}(\R^n  )\cap B_{\alpha+1}^{p(\alpha) ,\infty}(\R^n  )
\]
and
\[
 \|\widetilde{u}  \|_{L_{\alpha+1}^{p(\alpha),\infty}(\R^n )\cap B_{\alpha+1}^{p(\alpha),\infty}(\R^n )} \lesssim \|\widetilde{u}  \|_{L_{\beta+1}^{p(\beta) ,\infty}(\R^n  )} \approx \|u  \|_{L_{\beta+1}^{p(\beta),\infty}(\Omega )} \lesssim    \|f \|_{\mathcal{M}  (\Omega )}.
\]
This completes the proof of the theorem.
\end{proof}

\begin{rmk}\label{rep of measure as div}
It follows from Theorem  \ref{Poisson-measure data} that every $ f\in \mathcal{M}  (\Omega )$ can be written as
\[
f ={\rm div}\, F \quad\mbox{in}\,\, \Omega ,
\]
where $F =\nabla u$ for some $u$ with its zero extension $\widetilde{u}$ satisfying
\[
 \widetilde{u}   \in L_{\alpha+1}^{p(\alpha) ,\infty}(\R^n  )\cap B_{\alpha+1}^{p(\alpha) ,\infty}(\R^n  )\quad\mbox{for every}\,\, 0 \le \alpha < 1 .
\]
Since $\nabla$ is bounded from $L_{\alpha+1}^{p,q}(\R^n )$ into $L_{\alpha}^{p,q}(\R^n ;\R^n)$ and from $B_{\alpha+1}^{p,q}(\R^n )$ into $B_{\alpha}^{p,q}(\R^n ;\R^n )$  (see, e.g., \cite[(2.13) and (2.14)]{JK}), we have
\[
F  \in L_{\alpha}^{p(\alpha) ,\infty}(\Omega ;\R^n )\cap B_{\alpha}^{p(\alpha) ,\infty}(\Omega ;\R^n   )\quad\mbox{for every}\,\, 0 \le \alpha < 1 .
\]
\end{rmk}

\section{Proofs of Theorems \ref{Poisson-L1 data} and \ref{Poisson-div-L1 data-introd}}

To prove Theorems \ref{Poisson-L1 data} and \ref{Poisson-div-L1 data-introd}, we need to   recall the following results taken from \cite[Theorems 1.1 and 1.3]{KO}.

\begin{thm}\label{existence of solutions}
\begin{enumerate}[{\upshape(i)}]
\item  Let $\Omega$ be a bounded $C^1$-domain in $\R^n   $, $n\ge 2$. Suppose that $b \in L^{n,1}(\Omega ; \R^n )$,    $ c \in L^{n/2,1}(\Omega )\cap L^s (\Omega )$ for some $s>1$, and $c \ge 0$ in $\Omega$.
Then for every   $F \in L^{n/(n-1),\infty}(\Omega ;\R^n )$, there exists a unique weak  solution  $u$  of (\ref{eq:D-1-introd}) with $f={\rm div}\, F$ and $u_D =0$, which satisfies
\[
   \|   u \|_{L_{1}^{n/(n-1),\infty}(\Omega )}  \le C   \|F \|_{L^{n/(n-1),\infty}(\Omega ;\R^n  )}
\]
for some $C=C(n,  s, \Omega ,b,c )$.
\item Let $\Omega$ be a bounded $C^{1,1}$-domain in $\R^n , n \ge 2 $.  Suppose that $b \in L^{n,1}(\Omega ; \R^n )$,    $ c \in L^{n,1}(\Omega )$, and $c \ge 0$ in $\Omega$.
Then for every $ f\in L^{n,1}  (\Omega   )$, there exists a unique  strong solution  $u  $  in $L^{n,1}_2 (\Omega )$ of  (\ref{eq:D-1-introd}) with $u_D =0$, which satisfies
\[
 \|   u \|_{L_2^{n,1}(\Omega ) }  \le C   \|f \|_{L^{n,1}  (\Omega  )}
\]
for some  $C=C(n,  \Omega, b,c )$.
\end{enumerate}
\end{thm}

\subsection{Proof  of Theorem  \ref{Poisson-L1 data}}

Let $f \in L_{-1}^{n/(n-1),\infty}(\Omega )$ and $u_D \in \mathcal B^{n/(n-1),\infty}_{1/n}(\partial \Omega )$ be given. By Lemma \ref{trace results-1} (i), there exists  $u_1 \in L_{1}^{n/(n-1),\infty}(\Omega )$ such that
\[
{\rm Tr}\, u_1 = u_D \quad\mbox{and}\quad \|u_1 \|_{L_{1}^{n/(n-1),\infty}(\Omega )} \lesssim  \|u_D \|_{\mathcal B^{n/(n-1),\infty}_{1/n}(\partial \Omega )}.
\]
If $f_1 = f+  \Delta u_1 -b\cdot \nabla u_1 -cu_1$, then by H\"{o}lder's    inequality (\ref{Holder-L}), Lemma \ref{properties of SL on domains} (iii), and Lemma \ref{embedding of M},
\[
f_1 \in L_{-1}^{n/(n-1),\infty}(\Omega )\quad\mbox{and}\quad \|f_1 \|_{L_{-1}^{n/(n-1),\infty}(\Omega )} \lesssim   \|f\|_{L_{-1}^{n/(n-1),\infty}(\Omega )} + \|u_1 \|_{L_{1}^{n/(n-1),\infty}(\Omega )}  .
\]
By Theorem \ref{Poisson-measure data}, there exists  $F_1 \in L^{n/(n-1),\infty}(\Omega; \R^n )$  such that
\[
{\rm div}\, F_1 = f_1 \,\,\,\mbox{in}\,\,\Omega \quad\mbox{and}\quad \|F_1 \|_{L^{n/(n-1),\infty}(\Omega; \R^n )}   \lesssim\|f_1 \|_{L_{-1}^{n/(n-1),\infty}(\Omega )}.
\]
Hence by Theorem \ref{existence of solutions} (i),
there exists a unique   $u_2 \in L_{1, 0}^{n/(n-1),\infty}(\Omega )$ such that
\[
- \Delta u_2 + b \cdot \nabla u_2 +cu_2   = f_1   \quad   \text{in }\Omega.
\]
Moreover, this $u_2$ satisfies
\[
\|u_2 \|_{L_1^{n/(n-1),\infty}(\Omega  )} \lesssim  \|f_1 \|_{L_{-1}^{n/(n-1),\infty}(\Omega )} \lesssim \|f \|_{L_{-1}^{n/(n-1),\infty}(\Omega )} + \|u_D \|_{\mathcal B^{n/(n-1),\infty}_{1/n}(\partial \Omega )} .
\]
It is now easy to check that $u= u_1 +u_2$ is a unique solution in $ L_{1}^{n/(n-1),\infty}(\Omega )$ of (\ref{eq:D-1-introd}) and  satisfies
\[
   \|   u \|_{L_{1}^{n/(n-1),\infty}(\Omega )}  \le C \left( \|f \|_{L_{-1}^{n/(n-1),\infty}(\Omega )} + \|u_D \|_{\mathcal B^{n/(n-1),\infty}_{1/n}(\partial \Omega )} \right)
\]
for some $C=C(n,  s, \Omega ,b,c )$.

Suppose in addition that $ f\in \mathcal{M}  (\Omega )$ and $u_D \in \mathcal B^{p(\alpha ),\infty}_{\alpha +1 -1/p (\alpha )}(\partial \Omega )$ for   $0 <  \alpha <1$. Then $u_1$ can be chosen so that
\[
{\rm Tr}\, u_1 = u_D \quad\mbox{and}\quad \|u_1 \|_{L_{\alpha+1}^{p(\alpha),\infty}(\Omega )} \lesssim  \|u_D \|_{\mathcal B^{p(\alpha ),\infty}_{\alpha +1 -1/p (\alpha )}(\partial \Omega )}.
\]
Since  $f_1 = f+  \Delta u_1 -b\cdot \nabla u_1 -cu_1$, we have
\[
f_1 \in L_{\alpha-1}^{p(\alpha),\infty}(\Omega )\quad\mbox{and}\quad \|f_1 \|_{L_{\alpha-1}^{p(\alpha),\infty}(\Omega )} \lesssim   \|f\|_{L_{\alpha-1}^{p(\alpha),\infty}(\Omega )} + \|u_1 \|_{L_{\alpha+1}^{p(\alpha),\infty}(\Omega )}  .
\]
Moreover, if $f_2 = f_1 -b\cdot \nabla u_2 -cu_2$, then
\[
f_2 \in L_{\alpha-1}^{p(\alpha),\infty}(\Omega )\quad\mbox{and}\quad \|f_2 \|_{L_{\alpha-1}^{p(\alpha),\infty}(\Omega )} \lesssim   \|f_1 \|_{L_{\alpha-1}^{p(\alpha),\infty}(\Omega )} + \|u_2 \|_{L_1^{n/(n-1),\infty}(\Omega  )}  .
\]
Note that
$$
u_2 \in L_{1, 0}^{n/(n-1),\infty}(\Omega )\quad\mbox{and}\quad - \Delta u_2 = f_2 \,\,\,\mbox{in}\,\,\Omega.
$$
Therefore, by Theorem \ref{Poisson-FSL spaces} (i), we deduce that
\[
 u_2 \in L_{\alpha+1,0}^{p(\alpha) ,\infty}(\Omega)  \quad\mbox{and}\quad \| u_2 \|_{L_{\alpha+1}^{p(\alpha),\infty}(\Omega )}  \lesssim \|f_2 \|_{L_{\alpha-1}^{p(\alpha),\infty}(\Omega )}.
\]
This completes the proof of the theorem.
\qed

\subsection{Proof  of Theorem  \ref{Poisson-div-L1 data-introd}}

For the sake of simplicity, we write
$$
\mathcal{B}_\alpha = \mathcal{B}_{\alpha  - 1/p(\alpha)}^{p(\alpha ),\infty}(\partial\Omega ) = \left[ \mathcal{B}_{  1/p(\alpha)-\alpha}^{p(\alpha ) ' ,
1}(\partial\Omega ) \right]^*
$$
for $0 \le \alpha < 1$. Note that if $0 \le \beta <  \alpha < 1$, then $\mathcal{B}_\alpha \hookrightarrow \mathcal{B}_\beta$. Indeed,  since
 $$
(1-\beta)- (1-\alpha) = \frac{n}{p(\beta)'} -\frac{n}{p(\alpha)'} >0 ,
$$
it follows from Lemma \ref{properties of SL on domains} that
\[
\mathcal{B}_{  1/p(\beta)-\beta}^{p(\beta ) ' ,1}(\partial\Omega ) = {\rm Tr} \left(L_{1-\beta}^{p(\beta)' , 1}(\Omega )\right) \hookrightarrow {\rm Tr} \left(L_{1-\alpha}^{p(\alpha)' , 1}(\Omega )\right) =
\mathcal{B}_{  1/p(\alpha)-\alpha}^{p(\alpha ) ' ,1}(\partial\Omega ).
\]

Let $L$ be the differential operator defined by
\[
L \psi  = -\Delta \psi  +b \cdot \nabla \psi + c\psi  .
\]
Then it follows from Theorem \ref{existence of solutions} (ii)  that $L$ is an isomorphism from $L_{1,0}^{n,1}(\Omega )\cap L_2^{n,1}(\Omega)$ onto $L^{n,1}(\Omega )$.
Denote by $L^{-1}$ the inverse of $L$. Then for all $f \in L^{n,1}(\Omega )$,
\[
\|\nabla L^{-1}(f) \|_{C (\overline{\Omega}; \R^n )} \lesssim \|  L^{-1}(f) \|_{L_2^{n,1} (\Omega  )}\lesssim \|f \|_{L^{n,1} (\Omega )}.
\]
Hence given $G \in L^1 (\Omega ; \R^n )$ and  $v_D \in \mathcal{B}_0$, the mapping
\[
 f \ \mapsto \  \left\langle v_D ,   \partial_\nu L^{-1}(f)  \right\rangle - \int_\Omega G \cdot \nabla L^{-1}(f) \, dx
\]
defines a bounded linear functional   on $L^{n,1}(\Omega )$, with its norm being bounded above by $C \left(\|G\|_{L^1 (\Omega ; \R^n )} + \|v_D \|_{\mathcal{B}_0 } \right)$. Therefore, by the duality result (\ref{duality for L}), there exists a unique $v \in L^{n/(n-1),\infty}(\Omega )$, with
$$
\|v\|_{L^{n/(n-1),\infty}(\Omega )} \lesssim  \|G\|_{L^1 (\Omega ; \R^n )}
+ \|v_D \|_{\mathcal{B}_0} ,
$$
such that
\[
\int_\Omega vf \, dx =  \left\langle v_D ,   \partial_\nu L^{-1}(f)   \right\rangle - \int_\Omega G \cdot \nabla L^{-1}(f) \, dx
\]
for all $f\in L^{n,1}(\Omega )$. By the definitions of $L$ and $L^{-1}$, we also have
\[
\int_\Omega v\left( -\Delta \psi  +b \cdot \nabla \psi + c\psi\right)  dx =  \left\langle v_D ,   \partial_\nu \psi  \right\rangle - \int_\Omega G \cdot \nabla \psi  \, dx
\]
for all $\psi \in  L_{1,0}^{n,1}(\Omega )\cap L_2^{n,1}(\Omega)$. This implies, in particular, that $v$ is a   very weak solution   of  (\ref{eq:D-2-introd}) satisfying the desired estimate. Uniqueness of a very weak solution of (\ref{eq:D-2-introd}) can be easily proved by a duality argument based on the  density of $C_0 (\Omega )\cap C^{1,1}(\overline{\Omega})$ in $L_{1,0}^{n,1}(\Omega )\cap L_2^{n,1}(\Omega)$ (see \cite[Lemma 2.10 (iv)]{KO}).

To complete the proof, it remains to prove regularity results for the very weak solution $v$.
First using the result in  Remark \ref{rep of measure as div},
we choose $G_1 \in L^1 (\Omega ; \R^n )$ such that
\[
{\rm div}\, G_1 = cv \,\,\,\mbox{in}\,\,\Omega \quad \mbox{and}\quad \|G_1 \|_{L^1 (\Omega ;\R^n   )} \lesssim  \|c v\|_{L^1 (\Omega  )} \lesssim \|v\|_{L^{n/(n-1),\infty}(\Omega )} .
\]
Then defining $F  = G + v b -G_1 $, we have
\[
\|F \|_{L^1 (\Omega ;\R^n   )} \le C \left(\|G\|_{L^1 (\Omega ; \R^n )}
+ \|v_D \|_{\mathcal{B}_0 }  \right)
\]
for some $C=C(n,\Omega , b,c)$. Note that  $v $  is a very weak solution of the problem
\begin{equation}\label{Poisson with boundary data}
- \Delta v = {\rm div}\, F \,\,\,\mbox{in}\,\, \Omega
\quad\mbox{and}\quad v =v_D \,\,\,\mbox{on}\,\,\partial\Omega.
\end{equation}

Suppose now  that $v_D \in \mathcal{B}_\alpha$ for $0< \alpha < 1$.
Then since   $1 < p(\alpha)' < \infty$ and $ 1< 2-\alpha < 2$, it follows from   Theorem \ref{Poisson-FSL spaces} (ii) that  the solution operator $T$ of (\ref{eq:D-1}) is bounded from $L_{-\alpha}^{p(\alpha)',1}(\Omega )$ into $L_{1,0}^{p(\alpha)',1} (\Omega ) \cap  L_{2-\alpha, 0}^{p(\alpha)' ,1}(\Omega )$.
Moreover, since $(1-\alpha )p(\alpha)' =n$, it follows from Lemma \ref{properties of SL on domains} (iv)  that
\[
\|\nabla T(g)\|_{L^\infty  (\Omega; \R^n )} \lesssim  \|\nabla T(g)\|_{L_{1-\alpha}^{p(\alpha)'  ,1} (\Omega ; \R^n)}\lesssim  \| g\|_{L_{-\alpha}^{p(\alpha)'  ,1} (\Omega )}
\]
and
\[
\left| \left\langle v_D ,   \partial_\nu T(g)   \right\rangle \right| \le \|v_D \|_{\mathcal{B}_\alpha}\|\nabla T(g)\|_{L_{1-\alpha}^{p(\alpha)'  ,1} (\Omega ; \R^n)}\lesssim \|v_D \|_{\mathcal{B}_\alpha}\| g\|_{L_{-\alpha}^{p(\alpha)'  ,1} (\Omega )}
\]
for all $g \in L_{-\alpha}^{p(\alpha)',1}(\Omega )$. Hence the mapping
\[
g \ \mapsto \ \left\langle v_D ,   \partial_\nu T(g)   \right\rangle - \int_\Omega F \cdot   \nabla T(g) \, d x
\]
defines a bounded linear functional on $ L_{-\alpha}^{p(\alpha)',1}(\Omega )$ and its norm is bounded above by $C \left( \|F \|_{L^1  (\Omega ;\R^n )}+ \|v_D \|_{\mathcal B_\alpha } \right)$.
By Lemma \ref{real-interp-S on domains} (iv), 
there is a unique
$v_\alpha \in R_\Omega  (\widetilde{L}_\alpha^{p(\alpha),\infty} (\overline{\Omega} ) )  $,
with
$$
\|v_\alpha \|_{R_\Omega  (\widetilde{L}_\alpha^{p(\alpha),\infty} (\overline{\Omega} ) )} \lesssim \|F \|_{L^1  (\Omega ;\R^n )}+ \|v_D \|_{\mathcal B_\alpha},
$$
 such that
\begin{equation}\label{function u-alpha}
 \langle v_\alpha , g  \rangle = \left\langle v_D ,   \partial_\nu T(g)   \right\rangle - \int_\Omega F \cdot   \nabla T(g) \, d x \quad\mbox{for all} \,\, g \in L_{-\alpha}^{p(\alpha)',1}(\Omega ).
\end{equation}

Choose any $\beta$ with   $0 \le  \beta <  \alpha $.
Then since $\beta-0  = n/p(\beta)-n/p (0)$ and $\alpha  -\beta  = n/p(\alpha)  - n/p(\beta) >0$, it follows from Lemma \ref{properties of SL on Rn} (ii) and   Lemma \ref{embedding of L into B on domain} that
$$
\widetilde{v}_\alpha \in L_\alpha^{p(\alpha),\infty} (\R^n  ) \hookrightarrow
{L}^{n/(n-1),\infty}(\R^n )\cap {L}_{\beta}^{p(\beta),\infty}(\R^n )\cap {B}_{\beta}^{p(\beta),\infty}(\R^n )
$$
and
\begin{align*}
\|\widetilde{v}_\alpha \|_{{L}_{\beta}^{p(\beta),\infty}(\R^n )\cap {B}_{\beta}^{p(\beta),\infty}(\R^n )} & \lesssim
\|\widetilde{v}_\alpha \|_{L_\alpha^{p(\alpha),\infty} (\R^n  )}= \|v_\alpha \|_{R_\Omega  (\widetilde{L}_\alpha^{p(\alpha),\infty} (\overline{\Omega} ) )}\\
&\lesssim \|F \|_{L^1  (\Omega ;\R^n )}+ \|v_D \|_{\mathcal B_\alpha (\partial\Omega )}.
\end{align*}

Given $\psi \in C^{1,1} (\overline{\Omega} )$ with $\psi =0$ on $\partial\Omega$, we define  $g =-\Delta \psi$. Then since
$$
\psi , T(g) \in L_{1,0}^{p(\alpha)',1} (\Omega ) \cap  L_{2-\alpha, 0}^{p(\alpha)' ,1}(\Omega )\quad\mbox{and}\quad -\Delta \psi =g =-\Delta T(g)\,\,\mbox{in}\,\, \Omega,
$$
it follows from  Theorem \ref{Poisson-FSL spaces} (ii) that $\psi =T(g)$. Hence by  (\ref{function u-alpha}), we have
\[
 \int_\Omega v_\alpha (- \Delta
 \psi ) \, dx =\left\langle v_D ,   \partial_\nu  \psi   \right\rangle -  \int_\Omega F \cdot \nabla
 \psi\, d x ,
\]
which implies that $v_\alpha$ is a very weak solution  of  (\ref{Poisson with boundary data}). By the uniqueness of a very weak solution, we conclude that $v_\alpha =v$. This completes the proof of the theorem.
\qed

\section{Traces of      very weak solutions}

This final section is devoted to proving Theorem \ref{trace results-VWS}, which is a trace result  for     very weak solutions     of  (\ref{eq:D-2-introd}).
First we explore the trace of a  function   in $L^{p,q}_{2}(\Omega)$.

\subsection{Trace results for $L_2^{p,q} (\Omega )$}

Let $\Omega$ be a bounded $C^{1}$-domain  in $\R^n$.
Then for $1<p<\infty$ and $1 \le q \le \infty$, the operator $\widetilde{\rm Tr}$ defined by
\[
\widetilde{\rm Tr}\,u := \left( {\rm Tr }\, u ,\,  {\rm Tr} \,(\nabla u) \cdot \nu \right)
\]
is bounded from $L^{p,q}_{2}(\Omega)$ into
$\mathcal{B}^{p,q}_{2-1/p}(\partial \Omega)\times \mathcal{B}^{p,q}_{1-1/p}(\partial \Omega)$. Suppose in addition that $\Omega$ is  a   $C^{1,1}$-domain.
It is classical (see, e.g., \cite[Subsection 2.5.7]{Ne}) that for $k=1,2$, the range of the trace operator ${\rm Tr} : L_k^p (\Omega ) \to L^p (\partial\Omega )$ is the boundary Besov space $B_{k-1/p}^p (\partial \Omega )$; therefore,
\[
\mathcal{B}^{p,p}_{k-1/p}(\partial \Omega) = B_{k-1/p}^p (\partial\Omega ) \quad\mbox{for}\quad k=1,2 .
\]
It has been known (see, e.g.,  \cite[Subsection 2.5.7]{Ne}) that $\widetilde{\rm Tr} $   has a right inverse $\widetilde{\rm  Ex}$ that is bounded from  $B^{r}_{2-1/r}(\partial \Omega)\times B^{r}_{1-1/r}(\partial \Omega)$  into $L^{r}_{2}(\Omega)$ for all $1<r<\infty$.

\begin{lem}\label{trace results-2} Let  $\Omega$ be a  bounded $C^{1,1}$-domain in $\R^n$. Suppose  that   $1<p<\infty$ and    $1 \le q \le \infty$.
\begin{enumerate}[{\upshape(i)}]

\item  There is a bounded linear operator $\widetilde{\rm Ex}_1  : \mathcal B^{p,q}_{2-1/p}(\partial \Omega ) \rightarrow L_2^{p,q} (\Omega )$ such that ${\rm Tr} \circ \widetilde{\rm Ex}_1  $ is  the identity on $\mathcal B^{p,q}_{2-1/p}(\partial \Omega )$.
\item If  $1< p_0 \neq p_1 <\infty$,  $0< \theta <1$, and $1/p=(1-\theta)/p_0+\theta/p_1$,
then
\[
\left(  B^{p_0}_{2-1/p_0}(\partial \Omega ),  B^{p_1}_{2-1/p_1}(\partial \Omega ) \right)_{\theta,q}= \mathcal B^{p,q}_{2-1/p}(\partial \Omega ).
\]

\item 
    $\widetilde{\rm  Ex}$ is  bounded from $\mathcal B^{p,q}_{2-1/p}(\partial \Omega ) \times \mathcal  B^{p,q}_{1-1/p}(\partial \Omega )$ into $L_2^{p,q} (\Omega ) $
    and  
   $\widetilde{\rm Tr} \circ \widetilde{\rm  Ex} $ is  the identity on  $\mathcal  B^{p,q}_{2-1/p}(\partial \Omega ) \times \mathcal  B^{p,q}_{1-1/p}(\partial \Omega )$.

\end{enumerate}
\end{lem}
\begin{proof}
Suppose   that $1< p_0 \neq  p_1 <\infty$,  $0< \theta <1$, and $1/p=(1-\theta)/p_0+\theta/p_1$.  Then as in the proof of Lemma \ref{trace results-1} (ii), the trace operator        ${\rm Tr}$ is bounded from $L^{p,q}_{2}(\Omega)$ into $\left(  B^{p_0}_{2-1/p_0}(\partial \Omega ),  B^{p_1}_{2-1/p_1}(\partial \Omega ) \right)_{\theta,q}$ and
\[
 \mathcal B^{p,q}_{2-1/p}(\partial \Omega ) \hookrightarrow\left(  B^{p_0}_{2-1/p_0}(\partial \Omega ),  B^{p_1}_{2-1/p_1}(\partial \Omega ) \right)_{\theta,q}.
\]
On the other hand, the operator
\[
 f \mapsto \widetilde{\rm Ex}_1 \,( f) : = \widetilde{\rm  Ex} \, (f,0)
\]
is a  right inverse of  ${\rm  Tr} $ that is bounded from $B^{p_i}_{2-1/p_i}(\partial \Omega)$ into $L^{p_i}_{2}(\Omega)$  for $i=0,1$. By  real interpolation, $\widetilde{\rm Ex}_1$ is bounded from $\left(  B^{p_0}_{2-1/p_0}(\partial \Omega ),  B^{p_1}_{2-1/p_1}(\partial \Omega ) \right)_{\theta,q}$ into $L^{p,q}_{2}(\Omega)$ and
\[
\left(  B^{p_0}_{2-1/p_0}(\partial \Omega ),  B^{p_1}_{2-1/p_1}(\partial \Omega ) \right)_{\theta,q} \hookrightarrow \mathcal B^{p,q}_{2-1/p}(\partial \Omega ) ,
\]
which proves Parts (ii) and (i).  To prove Part (iii), we   observe from Part (ii), Lemma \ref{trace results-1}, and \cite[Chapter 3.13. Exercises 4]{bergh} that
\[\begin{split}
&\mathcal B^{p,q}_{2-1/p}(\partial \Omega ) \times \mathcal B^{p,q}_{1-1/p}(\partial \Omega )\\
&=\big(\, {B}^{p_0}_{2-1/p_0}(\partial\Omega),\, {B}^{p_1}_{2-1/p_1}(\partial\Omega)\,\big)_{\theta,q} \times \big(\, {B}^{p_0}_{1-1/p_0}(\partial\Omega),\, {B}^{p_1}_{1-1/p_1}(\partial\Omega)\,\big)_{\theta,q}\\
&=\big(\, {B}^{p_0}_{2-1/p_0}(\partial\Omega) \times {B}^{p_0}_{1-1/p_0}(\partial\Omega)\,,\,  {B}^{p_1}_{2-1/p_1}(\partial\Omega)\times {B}^{p_1}_{1-1/p_1}(\partial\Omega) \,\big)_{\theta,q}.
\end{split}\]
Hence   $\widetilde{\rm  Ex}$ is bounded from $\mathcal B^{p,q}_{2-1/p}(\partial \Omega ) \times \mathcal B^{p,q}_{1-1/p}(\partial \Omega )$ into $  L_2^{p,q}(\Omega )$.
This completes  the proof of the lemma.
\end{proof}

Consider the following Dirichlet problem for the bi-Laplace operator $\Delta^2$:
\begin{equation}\label{eq:D-2}
\left\{\begin{alignedat}{2}
 \Delta^2 u =  & \Delta  f & \quad & \text{in }\Omega,\\
  u   = \partial_\nu u & = 0 & \quad & \text{on }\partial\Omega  ,
\end{alignedat}\right.
\end{equation}
where $\Omega$ is a bounded $C^{1}$-domain  in $\R^n$.
Denote by $L^{p,q}_{2,(0,0)}(\Omega) $ or  $L^{p,q}_{2,{\bf 0}}(\Omega)$ the space of all $u\in L^{p,q}_{2}(\Omega)$ such that $ \widetilde{\rm Tr}\, u= (0,0)$ on $\partial\Omega$:
$$
L^{p,q}_{2,{\bf 0}}(\Omega) = \left\{u\in L^{p,q}_{2}(\Omega)\, :\, \widetilde{\rm Tr}\, u= (0,0) \text{ on }\partial\Omega \right\}.
$$
For $1<p=q<\infty$, we write $L^{p}_{2,{\bf 0}}(\Omega) = L^{p, p}_{2,{\bf 0}}(\Omega)$. Then it is known  (see, e.g., \cite[Subsection 2.5.7]{Ne} and  \cite[Proposition 1.14]{AP98}) that
\[
\widetilde{L}^{p}_{2}(\Omega) = L^{p}_{2 ,
{\bf 0}}(\Omega)  \quad \mbox{for}\,\, 1< p< \infty .
\]

\begin{lem}\label{biharmonic-S spaces}
Let $\Omega$ be a bounded $C^{1}$-domain in $\R^n , n \ge 2$. Suppose  that $1<p<\infty$ and $1 \le q \le \infty$.
Then for every $f\in L^{p,q}(\Omega )$, there exists a unique solution $u\in L_{2,{\bf 0}}^{p,q}(\Omega )$ of (\ref{eq:D-2}). Moreover,
\[
\|u\|_{L_{2}^{p,q}(\Omega )} \le C  \|f\|_{L^{p,q}(\Omega  )}
\]
for some $C=C(n,p,q,  \Omega )$.
\end{lem}

\begin{proof}  For the case when  $1<p=q<\infty$, this result was   shown  in \cite[Theorem 8.6]{DK2} for more general elliptic systems (see also \cite[Theorem 2.1]{AP98} and \cite[Theorem 1.6]{MM_book}).

For the general case, suppose that   $1<p_0 <  p_1 <\infty$, $0< \theta < 1$, and $1/p=(1-\theta)/p_0 +\theta/p_1$,
and $1 \le q \le \infty$.
Then it follows from  \cite[Theorem 8.6]{DK2}  that for every $f \in L^{p_0} (\Omega )$, there exists a unique solution $u \in L_{2,{\bf 0}}^{p_0} (\Omega )$ of  (\ref{eq:D-2}). Let   $\widetilde{T} : L^{p_0} (\Omega ) \to L_{2,{\bf 0}}^{p_0} (\Omega )$ be the associated  solution operator for (\ref{eq:D-2}).  Then  $\widetilde{T}$ is  also bounded from  $L^{p_1} (\Omega )$ into $ L_{2,{\bf 0}}^{p_1} (\Omega )$.
Therefore,   $T$ is bounded from $   L^{p,q}(\Omega ) $ into $\left(\,  L_{2,{\bf 0}}^{p_0}(\Omega ) , L_{2,{\bf 0}}^{p_1}(\Omega ) \, \right)_{\theta,q}$.
Moreover, it easily follows from Lemma \ref{interp for SL on domains} (ii) that $\left(\,  L_{2,{\bf 0}}^{p_0}(\Omega ) , L_{2,{\bf 0}}^{p_1}(\Omega ) \, \right)_{\theta,q} \hookrightarrow L_{2,{\bf 0}}^{p,q}(\Omega )$.
This completes the proof of the lemma.
\end{proof}

\begin{lem}\label{lem:interpolation_L00}
 Let $\Omega$ be a bounded $C^{1}$-domain in $\R^n , n \ge 2 $. Suppose  that $1<p_0 \neq  p_1 <\infty$, $0< \theta < 1$, $1/p=(1-\theta)/p_0 +\theta/p_1$,  and $1 \le q \le \infty$.
Then
\[
\left(\,  L_{2,{\bf 0}}^{p_0}(\Omega ) , L_{2,{\bf 0}}^{p_1}(\Omega ) \, \right)_{\theta,q} = L_{2,{\bf 0} }^{p,q}(\Omega ) .
\]
\end{lem}

\begin{proof}
By the proof of Lemma \ref{biharmonic-S spaces}, it remains to prove that
$$
L_{2,{\bf 0}}^{p,q}(\Omega ) \hookrightarrow \left(\,  L_{2,{\bf 0}}^{p_0}(\Omega ) , L_{2,{\bf 0}}^{p_1}(\Omega ) \, \right)_{\theta,q}.
$$
Suppose that  $u \in L_{2,{\bf 0}}^{p,q}(\Omega ) $ and $f= \Delta u$. Then $$
f \in  L^{p,q}(\Omega ) \quad\mbox{and}\quad \widetilde{T} f \in \left(\,  L_{2,{\bf 0}}^{p_0}(\Omega ) , L_{2,{\bf 0}}^{p_1}(\Omega ) \, \right)_{\theta,q} \hookrightarrow L_{2,{\bf 0}}^{p,q}(\Omega ) ,
$$
where  $\widetilde{T} $ is  the  solution operator for (\ref{eq:D-2}).
 Since $u, \widetilde{T} f \in L_{2,{\bf 0}}^{p,q}(\Omega )$ and
$\Delta^2 u =\Delta f  =\Delta^2  \widetilde{T} f$ in $\Omega$,
 it follows from Lemma \ref{biharmonic-S spaces} that
$u =\widetilde{T} f$. This completes the proof of the lemma.
\end{proof}

From Lemmas~\ref{real-interp-S on domains}  and \ref{lem:interpolation_L00}, we immediately obtain

\begin{lem}\label{R-tilde equal zero vector trac}  Let $\Omega$ be  a bounded $C^{1}$-domain in $\R^n , n \ge 2  $. If   $1<p<\infty$ and $1 \le q \le \infty$,  then
$R_\Omega  ( \widetilde{L}_2^{p,q}(\overline{\Omega} )) = L_{2, {\bf 0}}^{p,q}(\Omega )$. In addition, if $  q<\infty$, then
$\widetilde{L}_2^{p,q}(\Omega ) = L_{2, {\bf 0}}^{p,q}(\Omega ).$
\end{lem}

\subsection{Proof of Theorem  \ref{trace results-VWS}}

By a simple reduction as in the proof  of Theorem  \ref{Poisson-div-L1 data-introd}, we may assume  without loss of generality that $b =0$ and $c=0$, so that $v$ and $G$ satisfy
\begin{equation}\label{dist-formulation}
\int_\Omega v   \Delta \psi\,   dx =   \int_\Omega G \cdot \nabla \psi \, d x \quad\mbox{for all}\,\, \psi \in C_c^\infty (\Omega ).
\end{equation}

Note   that by Lemma \ref{trace results-2} (ii), the mapping $h\mapsto \widetilde{\rm  Ex}\, (0,h) $ defines a bounded linear operator $\widetilde{\rm  Ex}_2 : B_{1-1/n}^{n,1}(\partial\Omega )\rightarrow L_{1,0}^{n,1}(\Omega )\cap L_2^{n,1}(\Omega)$. Hence for  $h \in B_{1-1/n}^{n,1}(\partial\Omega ) $, we have
\begin{align*}
&\int_\Omega v \Delta (\widetilde{\rm  Ex}_2 h) \, dx - \int_\Omega G \cdot \nabla (\widetilde{\rm  Ex}_2 h) \, d x \\
&\quad \le \|v\|_{L^{n/(n-1),\infty}(\Omega )} \|\Delta (\widetilde{\rm  Ex}_2 h)\|_{L^{n,1}(\Omega )}+\|G\|_{ L^1 (\Omega ; \R^n)}\|\nabla (\widetilde{\rm  Ex}_2 h)\|_{L^\infty(\Omega )}\\
&\quad \lesssim  \left( \|v\|_{L^{n/(n-1),\infty}(\Omega )} +\|G\|_{ L^1 (\Omega; \R^n )} \right) \|\widetilde{\rm  Ex}_2 h\|_{L_2^{n,1}(\Omega )}\\
&\quad \lesssim  \left( \|v\|_{L^{n/(n-1),\infty}(\Omega )} +\|G\|_{ L^1 (\Omega; \R^n )} \right) \|  h\|_{B_{1-1/n}^{n,1}(\partial\Omega )}.
\end{align*}
Therefore, the  mapping
\[
h \mapsto \int_\Omega v \Delta (\widetilde{\rm  Ex}_2 h) \, dx - \int_\Omega G \cdot \nabla (\widetilde{\rm  Ex}_2 h) \, d x
\]
defines a bounded linear functional $\gamma_0 v$ on $B_{1-1/n}^{n,1}(\partial\Omega )$,  with its  norm being bounded above by $C \left( \|v\|_{L^{n/(n-1),\infty}(\Omega )}+ \|G\|_{ L^1 (\Omega ;\R^n )} \right)$. To prove (\ref{Green identity}), suppose that
$$
\psi \in L_{1,0}^{n,1}(\Omega )\cap L_2^{n,1}(\Omega) \quad\mbox{and}\quad h= {\rm Tr}\, (\nabla \psi) \cdot \nu.
$$
Then setting $\phi =\widetilde{\rm  Ex}_2 h$, we have
\[
\psi , \phi \in L_2^{n,1}(\Omega), \quad {\rm Tr} \, \psi ={\rm Tr} \, \phi =0, \quad  \mbox{and}\quad  {\rm Tr}\, (\nabla \psi) \cdot \nu = {\rm Tr}\, (\nabla \phi ) \cdot \nu =h ,
\]
which implies,  by Lemma \ref{R-tilde equal zero vector trac}, that
\[
\psi -\phi \in  \widetilde{L}_2^{n,1} (\Omega ).
\]
By a simple density argument, it follows from (\ref{dist-formulation}) that
\[
\int_\Omega v \Delta (\psi -\phi ) \, dx = \int_\Omega G \cdot \nabla (\psi-\phi) \, d x .
\]
Therefore,
\begin{align*}
\left\langle \gamma_0 v , {\rm Tr}\, (\nabla \psi) \cdot \nu \right\rangle   =
\left\langle \gamma_0 v , h \right\rangle
&=\int_\Omega v \Delta \phi \, dx - \int_\Omega G \cdot \nabla \phi \, d x\\
&=\int_\Omega v \Delta \psi \, dx - \int_\Omega G \cdot \nabla \psi \, d x
.
\end{align*}
This proves the existence of $\gamma_0 v$ satisfying the desired  estimate of the theorem.

To prove the uniqueness assertion, suppose that $f$ is any bounded linear functional   on $B_{1-1/n}^{n,1}(\partial\Omega )$ satisfying \begin{equation}\label{Green identity-2}
\left\langle f, {\rm Tr}\, (\nabla \psi) \cdot \nu \right\rangle =\int_\Omega v \Delta \psi \, dx - \int_\Omega G \cdot \nabla \psi \, d x
\end{equation}
for all $\psi \in C^{1,1}(\overline{\Omega} )$ with $\psi =0$ on $\partial\Omega$. Then it follows from a density result in \cite[Lemma 2.10 (iv)]{KO} that (\ref{Green identity-2}) also  holds for all $\psi \in L_{1,0}^{n,1}(\Omega )\cap L_2^{n,1}(\Omega)$.
Given  $h \in B_{1-1/n}^{n,1}(\partial\Omega )$, let $\psi :=\widetilde{\rm  Ex}_2 h\in L_{1,0}^{n,1}(\Omega )\cap L_2^{n,1}(\Omega)$. Then since ${\rm Tr}\, (\nabla \psi) \cdot \nu=h$, it follow from  (\ref{Green identity-2}) and the definition of $\gamma_0 v$ that
\[\begin{split}
\left\langle f , h \right\rangle
&=\int_\Omega v \Delta \psi  \, dx - \int_\Omega G \cdot \nabla \psi  \, d x\\
&= \int_\Omega v \Delta (\widetilde{\rm  Ex}_2 h) \, dx - \int_\Omega G \cdot \nabla (\widetilde{\rm  Ex}_2 h) \, d x
 =\left\langle \gamma_0 v ,  h \right\rangle ,
\end{split}\]
which implies that $f =\gamma_0 v$.

Suppose in addition that $v \in L_1^{n/(n-1),\infty}(\Omega )$. Then by (\ref{dist-formulation}), we easily obtain
\begin{equation}\label{weak so-VWS}
\int_\Omega \left( \nabla v+ G \right)\cdot \nabla \psi \, d x =0 \quad\mbox{for all}\,\,\psi \in L_{1,0}^{n,1}(\Omega ).
\end{equation}
Given  $h \in B_{1-1/n}^{n,1}(\partial\Omega )$, let $\psi :=\widetilde{\rm  Ex}_2 h\in L_{1,0}^{n,1}(\Omega )\cap L_2^{n,1}(\Omega)$. Then by the divergence theorem and (\ref{weak so-VWS}),
\begin{align*}
\left\langle \gamma_0 v, h   \right\rangle &=\int_\Omega v \Delta \psi \, dx - \int_\Omega G \cdot \nabla \psi \, d x \\
&= \int_{\partial \Omega} ({\rm Tr}\, v) h \, d \sigma   - \int_\Omega \left( \nabla v+ G \right)\cdot \nabla \psi \, d x = \int_{\partial \Omega} ({\rm Tr}\, v) h \, d \sigma.
\end{align*}
This completes the proof of the theorem.
\qed

\bigskip
\noindent

\textbf{Acknowledgements: }
This work was supported by the Sogang University Research Grant of 2023 (202311003.01).
The authors were partially supported by the National Research Foundation of Korea(NRF) grant funded by the Korea government: H. Kim  by the Ministry of Education (No. NRF-2016R1D1A1B02015245); Y.-R. Lee by   the Ministry of Science and ICT (No. RS-2020-NR048589); J. Ok by the Ministry of Science and ICT (No. NRF-2022R1C1C1004523).


\end{document}